\documentclass[10pt,twoside]{amsart}
  \usepackage{amsmath,amssymb,amsthm,amscd,latexsym,graphicx}
  \usepackage{hyperref}
  
 \usepackage{epigraph}
 \usepackage[OT2,OT1]{fontenc}
  \usepackage{enumitem}
\usepackage{tikz}
\usepackage{graphicx}
\usepackage[all]{xy}

 \font\caps=cmcsc10     
 \font\Caps=cmcsc10 scaled \magstep1 

 \entrymodifiers={!!<0pt,0.7ex>+}

 \makeatletter
 \@specialpagefalse\headheight=8.5pt
 \makeatother

 \def\TSkip{\medskip}
 \newbox\TheTitle{\obeylines\gdef\GetTitle #1
 \ShortTitle #2
 \SubTitle #3
 \Author  #4
 \ShortAuthor #5
 \EndTitle
 {\setbox\TheTitle=\vbox{\baselineskip=20pt\let\par=\cr\obeylines%
 \halign{\centerline{\Caps##}\cr\noalign{\medskip}\cr#1\cr}}%
   \copy\TheTitle\TSkip\TSkip%
 \def\next{#2}\ifx\next\empty\gdef\STitle{#1}\else\gdef\STitle{#2}\fi%
 \def\next{#3}\ifx\next\empty%

 \else\setbox\TheTitle=\vbox{\baselineskip=20pt\let\par=\cr\obeylines%
  \halign{\centerline{\caps##} #3\cr}}\copy\TheTitle\TSkip\TSkip\fi%
 \centerline{\caps #4}\TSkip\TSkip%
 \def\next{#5}\ifx\next\empty\gdef\SAuthor{#4}\else\gdef\SAuthor{#5}\fi%
 \catcode'015=5}}

 \long\def\MSC#1\EndMSC{\def\arg{#1}\ifx\arg\empty\relax\else
  {\par\narrower\noindent%
  2000 Mathematics Subject Classification: #1\par}\fi}

 \long\def\KEY#1\EndKEY{\def\arg{#1}\ifx\arg\empty\relax\else
   {\par\narrower\noindent Keywords and Phrases: #1\par}\fi\TSkip}

 \long\def\DATE#1\EndDATE{\def\arg{#1}\ifx\arg\empty\relax\else
   {\par\narrower\noindent \center{\textit{#1}}\par}\fi\TSkip\TSkip\TSkip}

 \renewcommand{\U}{\mathbf{U} }

\newcommand{\obs}{ \mathrm {obs}}

  \DeclareSymbolFont{cyrletters}{OT2}{wncyr}{m}{n}
   \DeclareMathSymbol{\Sha}{\mathalpha}{cyrletters}{"58}

    \newcommand{\nat}{\mathrm{nat}}

 \renewcommand{\to}{\longrightarrow}

 \newcommand{\cE}{\mathcal E}

 \newcommand{\C}{\mathbb{C}}
 \newcommand{\F}{\mathbb{F}}
 \newcommand{\G}{\mathbf{G}}

 \newcommand{\N}{\mathbb{N}}
 \renewcommand{\P}{\mathbb P}
 \newcommand{\Q}{\mathbb{Q}}
 \newcommand{\R}{\mathbb{R}}
 \newcommand{\Z}{\mathbb{Z}}

\renewcommand{\2}{{\mathbb {Z} / 2 \mathbb Z}}

 \newcommand{\Res}{\mathrm{Res}}
  \newcommand{\Inf}{\mathrm{Inf}}

 \newcommand{\GL}{\mathbf{GL}}

 \newcommand{\Aut}{\mathbf{Aut}}
 \newcommand{\End}{\mathbf{End}}

 \newcommand{\Br}{\mathrm{Br}}
 
 \newcommand{\ovl}{\overline}
 \newcommand{\vareps}{{\varepsilon}}

 \newcommand{\Ext}{\mathrm{Ext}}
\newcommand{\Gal}{\mathrm{Gal}}
 
 \newcommand{\Hom}{\mathrm{Hom}}
 \newcommand{\im}{\mathrm{im}}
 \newcommand{\Ver}{\mathrm{Ver}}
 \newcommand{\Frob}{\mathrm{Frob}}

 \newcommand{\B}{\mathbf{B}}

\newcommand{\Id}{\mathrm{Id}}

\newcommand{\Ker}{\mathrm{Ker}}

\newcommand{\W}{\mathbf{W}}

 \theoremstyle{plain}
 \newtheorem{thm}{Theorem}[section]
   \newtheorem*{thm*}{Theorem}
 \newtheorem{defi}[thm]{Definition}

 \newtheorem{prop}[thm]{Proposition}

 \newtheorem{lem}[thm]{Lemma}
  
 \newtheorem{coro}[thm]{Corollary}
 \newtheorem*{coro*}{Corollary}

 \newtheorem*{theorem-non}{Bloch-Kato conjecture, an equivalent formulation}
\newtheorem*{thmA*}{Theorem A}
\newtheorem*{thmB*}{Theorem B}
\newtheorem*{thmC*}{Theorem C}
\newtheorem*{thmD*}{Theorem D}

 \theoremstyle{remark}
 \newtheorem{rem}[thm]{Remark}
  \newtheorem*{rem*}{Remark}

 \newtheorem{ex}[thm]{Example}
 
 \newtheorem{exo}[thm]{Exercise}

 \newenvironment{dem}{{\bf Proof.}}{\hfill$\square$}

\begin{document}


 \title{Lifting Galois representations via Kummer flags}

 \author{Andrea Conti, Cyril Demarche and Mathieu Florence}

\address{Andrea Conti, Heidelberg University, 69120 Heidelberg, Germany \\
Cyril Demarche and Mathieu Florence, Sorbonne Université and Université Paris Cité, CNRS, IMJ-PRG, F-75005 Paris, France.}

\begin{abstract}

Let  $\Gamma$ be either  i) the absolute Galois group of a local field $F$, or  ii) the topological fundamental group of a closed connected orientable surface of genus $g$. In case i), assume that $\mu_{p^2} \subset F$ . We give an elementary and unified proof that every   representation  $\rho_1: \Gamma \to \GL_d(\F_p)$ lifts to a representation  $\rho_2: \Gamma \to \GL_d(\Z/p^2)$. [In case i), it is understood these are  continuous.] The actual statement is   much stronger: for all $r \geq 1$, assuming $\mu_{p^{r+1}} \subset F$ in the local field case, strictly upper  triangular representations  $\rho_r: \Gamma \to \U_d(\Z/p^r)$ lift  to  $\rho_{r+1}: \Gamma \to \U_d(\Z/p^{r+1})$, in the strongest possible step-by-step sense. Here ``suitable" is made precise by the concept of  \textit{Kummer flag}. An essential aspect of this work is to identify  the common properties of groups i) and ii), that suffice to ensure the existence of such lifts.
\end{abstract}
\maketitle
\newpage
\tableofcontents

\section{Introduction}
Let $p$ be a prime. 
This paper presents a common framework to tackle lifting problems (from $\Z/p^r$ to $\Z/p^{r+1}$) for representations of two kinds of groups $\Gamma$: absolute Galois groups of local fields (which we refer to as the arithmetic case), and topological fundamental groups of closed connected orientable surfaces (the topological case). 
The challenge here is to achieve a reasonable degree of generality (say, in the arithmetic case, compared to the results of \cite{BIP} and \cite{EG}), with a light toolkit and little effort. 

Using local class field theory in the arithmetic case, and Poincar\'e duality in the topological case -- both being very standard results -- we derive common properties of groups of the above type. Namely, in both cases, the group $\Gamma$ is \textit{$p$-manageable} (Definition \ref{defpmana}). This notion combines the concepts of a (possibly non-finitely generated) Demu$\check{\textup{s}}$kin group (Definition \ref{Defidemu}), and a weak version of Hilbert's Theorem 90, called the wH90 property (Definition \ref{def wH90}). Our lifting theorems are valid for all $p$-manageable profinite groups. These include two-dimensional Poincaré groups, see Section \ref{PDG}. Via basic operations for extensions of representations of (pro-)finite groups, we are able to lift a reducible $\Z/p^r$-representation $\rho_r$ of $\Gamma$ to a $\Z/p^{r+1}$-representation, whenever we can equip $\rho_r$ with a \emph{Kummer flag}: loosely speaking, this means that $\rho_r$ is unipotent, and every sub-extension of $\rho_r$ which is trivial mod $p$ is also trivial mod $p^r$. This is a combinatorial property, for which we refer to Definition \ref{def kummer}. 

We give a more precise statement of our results. To fix ideas, let $\Gamma$ be one of the following groups: 
\begin{itemize}
	\item the absolute Galois group of a local field $F$ ($\R$, or a finite extension of $\Q_\ell$ or $\F_\ell((t))$, with $\ell=p$ allowed);
	\item the topological fundamental group of a closed connected orientable surface.
\end{itemize}

Then our Theorem \ref{ThmLiftK} implies the following statement. Observe that the assumption on roots of unity is necessary here -- as demonstrated in Example \ref{ex5dim}.

\begin{thm*}
Let $d,r \geq 1$ be two integers, and let $\rho_r\colon\Gamma\to\GL_d(\Z/p^r)$ be a continuous representation equipped with a Kummer flag. In the arithmetic case, assume that $F$ contains a primitive $p^{r+1}$-th root of unity.

Then $\rho_r$ admits a lift $\rho_{r+1}\colon\Gamma\to\GL_d(\Z/p^{r+1})$, equipped with a Kummer flag lifting the given one on $\rho_r$.
\end{thm*}

For mod $p$ unipotent representations, there is the following immediate corollary. Here again, the assumption on roots of unity is necessary.

\begin{coro*}
Let $\rho_1 \colon\Gamma\to\U_d(\Z/p) \subset \GL_d(\Z/p)$ be a strictly upper triangular representation. In the arithmetic case, assume that $F$ contains a primitive $p^{r+1}$-th root of unity. Then $\rho_1$ lifts, to a strictly upper triangular representation   $\rho_{r+1}\colon\Gamma\to\U_d(\Z/p^{r+1}) \subset \GL_d(\Z/p^{r+1})$.
\end{coro*}

The proof of Theorem \ref{ThmLiftK} goes by induction on the dimension of the representation, and actually gives a much finer result: given a $(d-1)$-dimensional Kummer flag $\nabla_{d-1,r}$ mod $p^r$, an extension of $\nabla_{d-1,r}$ to a $d$-dimensional mod $p^r$ Kummer flag $\nabla_{d,r}$, and a lift $\nabla_{d-1,r+1}$ of $\nabla_{d-1,r}$ to a $(d-1)$-dimensional Kummer flag mod $p^{r+1}$, we can ``glue'' $\nabla_{d,r}$ and $\nabla_{d-1,r+1}$ along $\nabla_{d-1,r}$ to obtain a $d$-dimensional Kummer flag mod $p^{r+1}$ that lifts $\nabla_{d,r}$ and extends $\nabla_{d-1,r+1}$ (a process that we refer to as ``gluifting''). 
This stronger statement has the following cohomological consequence.

\begin{coro*}(See Corollary \ref{CoroLiftK}) \\
Let $r \geq 2$ be an integer. Assume, in the arithmetic case, that $F$ contains a primitive $p^r$-th root of unity. Let $V_1$ be a unipotent representation of $\Gamma$ over $\F_p$. There exists a lift of $V_1$, to a unipotent representation $V_r$ over $\Z/p^r$, such that the natural map  \[ H^1(\Gamma, V_r) \to  H^1(\Gamma, V_1=V_r/p)\] is surjective.
\end{coro*}

Via classical cohomological arguments, our main theorem recovers the liftability of arbitrary (=not necessarily triangularizable) representations from $\F_p$ to $\Z/p^2$.

\begin{coro*}(See Corollary \ref{CoroLift})\\
   Let $d \geq 1$ be an integer, and let $\rho_1: \Gamma \to \GL_d(\F_p)$ be a mod $p$ representation. 
    In the arithmetic case, assume that $F$ contains a primitive $p^2$-th root of unity. Then $\rho_1$ lifts to a representation $\rho_2: \Gamma \to \GL_d(\Z/p^2).$
\end{coro*}

\begin{rem}
This Corollary actually holds without assumption on roots of unity, by   \cite{BIP} or \cite{EG}. In our approach, it is not clear to us how to get rid of it.
\end{rem}
\begin{rem*}
    If $F$ is a field, and $k \geq 1$, let $F(\mu_k)$ denote the extension of $F$ generated by $k$-th roots of unity. 
  In the previous theorem and its corollaries,  one can replace the condition about roots of unity (in the arithmetic case) by the slightly more general assumption that $F(\mu_p) = F(\mu_{p^{r+1}})$, where $r=1$ in the last corollary (see Remark \ref{rem kummer general}).
\end{rem*}

In addition, we prove strong results for $\rho_r$ admitting a unique complete $\Gamma$-invariant flag, with no assumption on roots of unity in the arithmetic case. See Proposition \ref{propploye}, where we call such a unique flag a \emph{wound Kummer flag}.\\
In all of our results, $\Z/p^r$ can be replaced with the ring of Witt vectors $\W_r(k)$ of a  field $k$ of characteristic $p$ (see Section \ref{SectGen}). For simplicity, we  stick to $k=\Z/p$. \\We remark in particular that \textit{we are able to control the coefficients of the lifts,} which is not the case in previous arithmetic results that we recall below. 

On the topological side, as far as we know, our results are entirely new. In the course of an email exchange, Hélène Esnault observed that the case of an absolutely irreducible $\rho_1$, in Item (2), is simple to handle. Indeed, the corresponding $\F_p$-point of the character variety of $\Gamma$ is then smooth. As such, it lifts to a $\Z_p$-point, producing a  lift of $\rho_1$ to $\rho_\infty: \Gamma \to \GL_d(\Z_p) .$ This fact illustrates  that,  with regards to liftability, the (opposite) case of unipotent representations is the most delicate.
\subsection*{Relating our results to other works, on the arithmetic side}\hfill\\
\noindent In this number-theoretic context,  liftability of a mod $p$ representation, to coefficients in $\Z_p$ or a DVR containing it, has  been widely investigated. One typically requires such a lift  to possess certain $p$-adic Hodge theoretic properties (i.e. crystalline, semistable, de Rham) so that one can (conjecturally) attach to it an automorphic representation of a suitable reductive group. We refer to the introduction of \cite{KL} for a brief summary of this line of work, prior to the very recent works of Emerton-Gee \cite{EG} and B\"ockle-Iyengar-Pa\v{s}k\={u}nas \cite{BIP}. To conclude, we  briefly  relate our work to theirs, as well as to earlier works of B\"ockle, Clozel-Harris-Taylor, and Khare-Larsen. \\

Emerton and Gee \cite{EG} rely on advanced techniques in arithmetic geometry to prove that every continuous representation $\ovl\rho\colon \Gamma_F\to\GL_d(\F_p)$, $F$ a $p$-adic field, admits a crystalline lift $\rho\colon \Gamma_F\to\GL_d(\mathcal O)$, $\mathcal O$ the valuation ring of a $p$-adic field. Their proof goes via a study of the geometry of the stack of $(\varphi,\Gamma)$-modules that they construct. 
Specializing $\rho$ modulo $p^r$, $r\ge 1$, gives a mod $p^r$ lift of $\rho_1$. 
When $\rho_1$ is triangular (i.e. admits a complete flag of sub-representations), the Emerton-Gee lift $\rho$ is also triangular, since it is constructed via recursive liftings of the Jordan-H\"older factors of $\rho_1$ and of their extensions \cite[Theorem 6.4.4]{EG}. In particular, they produce triangular mod $p^r$ lifts of $\rho_1$, as we do in the case when $F$ contains the $p^r$-th roots of unity. However, their lifting result does not imply ours: in Theorem \ref{ThmLiftK}, for $r \geq 2$, one starts with a mod $p^r$ representation $\rho_r$, equipped with a Kummer flag, and one constructs a mod $p^{r+1}$ lift, also equipped with a Kummer flag. Emerton and Gee's method, on the other hand, does not directly produce a lift of $\rho_r$: one first needs to reduce $\rho_r$  mod $p$, and then lift (mod $p^{r+1}$) the resulting $\rho_1$ to a $\tilde \rho_{r+1}$ with $\mathcal O$-coefficients. Its mod $p^r$ reduction is in general  not  isomorphic to $\rho_r$. Also, it is not guaranteed that a given  sub-extension of  $\tilde \rho_{r+1}$ is trivial, as soon as the corresponding sub-extension of $\rho_1$ is. \\

B\"ockle, Iyengar and Pa\v{s}k\={u}nas \cite{BIP} also prove the existence of  lifts of an arbitrary $\rho_1$  to $\mathcal O$-coefficients as above (hence to $\mathcal O/p^r$), via deformation theory. However, they do not mention that if $\rho_1$ is triangular then one can always find a triangular $\Z/p^r$-lift in its deformation space.\\ 

For an arbitrary local field $F$, and for global $F$ in some cases when the corresponding deformation problem is unobstructed, B\"ockle \cite{Bo} is able to lift continuous $\Z/p$-linear representations of $\Gamma_F$ to $\Z/p^2$, using a description of the absolute Galois group via generators and relations. Therefore,  Corollary \ref{CoroLift}(2)  follows from his work.  As far as explicitly shown, B\"ockle's method produces lifts to $\Z_p$ (hence to $\Z/p^r$, $r>2$) only in  special cases -- see \cite[Theorem 1.3, Proposition 1.4]{Bo}. \\

Suppose that $F$ is a local field of characteristic 0 and residual characteristic $\ell\neq p$. Then Clozel, Harris and Taylor \cite[Section 2.4.4]{CHT} prove that every $\rho_1$ admitting a complete flag lifts to a $\Z_p$-representation also admitting a complete flag, with  no assumption on roots of unity. \\

Finally, Khare and Larsen \cite[Theorem 5.4]{KL} prove the following special case  of our Corollary \ref{CoroLift}: $d=3$, $r=2$, and the diagonal is trivial mod $p$. Then, by an approximation process, they obtain the same lifting result for  global fields. We did not try to adapt our results to the global arithmetic case. Hopefully this will be done in future works.

\setcounter{tocdepth}{1}
\tableofcontents

    \section{Notation and conventions}
Let $\Gamma$ be a profinite group and $p$ be a prime number. For an integer $d\geq 1$, denote by $\B_d \subset \GL_d$ the subgroup formed by upper triangular matrices, and by $\U_d\subset\B_d$ the subgroup of $\B_d$ consisting of  unipotent matrices. \\
Let $r\geq 1$ be an integer. A $(\Gamma,r)$-module is a finite $(\Z/p^r)$-module $M$, equipped with a continuous $\Gamma$-action. Here ``continuous'' just means that $\Gamma $ acts on $M$ through a finite quotient $\Gamma/\Gamma_0$, where $\Gamma_0 \subset \Gamma$ is an open normal subgroup. Set $$M^\vee:=\Hom_{\Z/p^r}(M,\Z/p^r),$$ viewed as $(\Gamma,r)$-module in the natural way.  If $M$ is moreover free of rank $d$ as a $(\Z/p^r)$-module, we say that $M$ is a $(\Gamma,r)$-bundle of rank $d$. Then, $M^\vee$ is a $(\Gamma,r)$-bundle of rank $d$, as well.

 Let $M_r$ be  a $(\Gamma,r)$-bundle. For $1 \leq s < r $, there is  the natural \textit{reduction} exact sequence of  $(\Gamma,r)$-modules   \[0 \to M_{r-s} \stackrel {i} \to  M_r \stackrel q \to M_s \to 0. \]  Here $i$ denotes the injection arising from $ M_r \xrightarrow{p^s \Id } M_r$ upon modding out $p^{r-s} M_r$, and $q$ stands for the natural quotient (reduction).\\
\subsection{Flags} A complete $(\Gamma,r)$-flag $\nabla_d = (V_i)_{1 \leq i \leq d}$ of rank $d$ is the data of a $(\Gamma,r)$-bundle $V_d$, of rank $d$, together with a complete filtration by sub-$(\Gamma,r)$-bundles $$0=V_0 \subset V_1 \subset \dots \subset V_{d-1} \subset V_d,$$  with $\mathrm{rank}(V_i)=i$ for all $i$. Thus, for all $i$, $L_i := V_i / V_{i-1}$ is  a $(\Gamma,r)$-bundle of rank one. Up to isomorphism, it is given by  a character $\chi_i : \Gamma \to(\Z/p^r)^\times$. For $i \leq j,$ set $$ V_{j/i}:=V_j/V_i.$$

In particular, a complete $(\Gamma,r)$ flag defines,  for all $1 \leq i \leq d$,   a character $$\chi_i : \Gamma \to \GL(L_i) = (\Z/p^r)^\times.$$ 

We denote by $\End(\nabla_d)$ the $(\Gamma,r)$-module of endomorphisms of $\nabla_d$. It is the  sub-ring of $\End(V_d)$ that preserves the given complete filtration. Write $\Aut(\nabla_d)$ for $\End(\nabla_d)^\times$, the group of invertible elements in the ring $\End(\nabla_d)$.

For $0 \leq i<j<k \leq d$, there are extensions of $(\Gamma,r)$-modules 
\[\mathcal E_{k/j, j/i}: 0 \to V_{j/i} \xrightarrow{\iota} V_{k/i} \xrightarrow{\pi} V_{k/j} \to 0 \, ,\] naturally attached to $\nabla_d$.

Denote by $$\nabla_{d-1}: 0=V_0 \subset V_1 \subset \dots \subset V_{d-1}, $$    resp. by $$\nabla_{d/1}: 0=V_0 \subset V_{2/1} \subset V_{3/1} \subset \dots \subset V_{d/1}, $$  the truncation, resp. the quotient, of $\nabla_d$. These are complete $(d-1)$-dimensional $(\Gamma,r)$-flags. \\
A complete $(\Gamma,r)$-flag $\nabla_d$ can be seen as a triangular representation of $\Gamma $ over $(\Z/p^r)$: given a basis of $V_d$ compatible with $\nabla_d$, one gets a representation
\[\rho : \Gamma \to \B_d(\Z/p^r) \subset \GL_d(\Z/p^r). \]
Note that $\rho$ factors through $\U_d(\Z/p^r) \subset \B_d(\Z/p^r)$, if and only if the $\Gamma $-action on every $L_i$  is trivial, or equivalently if every $\chi_i=1$.

From now on, the notation $V_{i,r}$ stands for a $(\Gamma,r)$-bundle of rank $i$, and accordingly, $\nabla_{d,r}$ denotes a complete $d$-dimensional flag of $(\Gamma,r)$-bundles. If $V_{i,r}$ is given, then for all $1 \leq s \leq r$,  we denote its reduction  $V_{i,r} \otimes_\Z \Z/p^s$ by $V_{i,s}$. It is  a $(\Gamma,s)$-bundle. The  similar convention for complete flags   $\nabla_{d,r}$ is adopted.\\ 

\subsection{Teichm\"uller lift}
Let $L$ be an invertible $(\Gamma,1)$-bundle, given by a character $$ \chi_L: \Gamma \to \F_p^\times .$$
 Consider the multiplicative (Teichm\"uller) section $$\tau: \F_p^\times \to (\Z/p^r)^\times ,$$ identifying $\F_p^\times$ with the group of $(p-1)$-th roots of unity in the ring $\Z/p^r$.\\ Postcomposing  by $\tau$ yields a character $$ \tau(\chi_L): \Gamma \to (\Z/p^r)^\times  ,$$ providing a natural $(\Gamma,r)$-bundle of rank one,  denoted by $\W_r(L)$. It is the $r$-th Teichm\"uller lift of $L$.

\section{Galois cohomology and $p$-manageable groups}
Let $p$ be a prime.
In this section, we recall the notion  of  a (non finitely-generated) Demu$\check{\textup{s}}$kin pro-$p$-group. We then introduce the wH90 property, and $p$-manageable profinite groups. We state in particular Lemma \ref{LemManaLift}, the key tool for proving our lifting theorems in the next sections. These apply to all $p$-manageable profinite groups. In Proposition  \ref{Prop2G}, we provide two important examples of such groups: the absolute Galois group of a local field, and the profinite completion of the fundamental group of a closed connected orientable surface.
\begin{defi}
    Let $$ (.,.): V \times W \to \F_p$$ be a bilinear pairing of $\F_p$-vector spaces. Consider it as a linear map $$ L: V \to W    ^\vee,$$ $$v \mapsto (w \mapsto (v,w)).$$ The left kernel of $ (.,.)$  is the subspace $\Ker(L) \subset V$.
\end{defi}

\begin{defi}\label{Defidemu}

Let $\Gamma_p$ be a pro-$p$-group. Let us say that $\Gamma_p$ is a Demu$\check{\textup{s}}$kin pro-$p$-group, if the following conditions are satisfied. \begin{enumerate}

     \item {The $\F_p$-vector space $H^2(\Gamma_p,\F_p)$ is one-dimensional.}  \item {The cup-product  pairing  of possibly infinite-dimensional $\F_p$-vector spaces \[H^1(\Gamma_p,\F_p) \times H^1(\Gamma_p,\F_p) \to H^2(\Gamma_p,\F_p) \] has trivial (left) kernel.}

\end{enumerate}

More generally, if $\Gamma$ is a profinite group, one says that $\Gamma$ is $p$-Demu$\check{s}$kin, if it admits a Demu$\check{s}$kin pro-$p$-Sylow subgroup $\Gamma_p \subset \Gamma$.
\end{defi}

\begin{rem}
 In the classical literature, Demu$\check{\textup{s}}$kin pro-$p$-groups $\Gamma_p$ are  required to be topologically finitely generated -- a condition that is equivalent to the finiteness of $H^1(\Gamma_p,\F_p)$. However, this assumption is superfluous to prove the main lifting results of this paper. For a nice investigation of (not necessarily topologically finitely generated) Demu$\check{\textup{s}}$kin pro-$p$-groups, we refer to the recent work \cite{BN}.
\end{rem}

\begin{rem}
    If  $\Gamma$ is the absolute Galois group of a ``reasonable'' field $F$, condition (2) is very mild -- see Lemma \ref{Lem2FT}. By contrast, condition (1) is extremely strong. It does not typically hold for an $F$ which is finitely generated over its prime subfield. For instance, if $\mu_p \subset F$, then $H^2(\Gamma,\Z/p)$ is isomorphic to the $p$-torsion in the Brauer group of $F$.
\end{rem}
\begin{rem}

For $p=2$, the group $\Gamma:=\Z/2\Z$ is easily checked to be Demu$\check{\textup{s}}$kin. This fails for  $p \geq 3$: the group $\Gamma:=\Z/p\Z$ satisfies condition (1), but not (2). Indeed, the cup-product pairing then identically vanishes. These assertions are straightforward, from the usual computation of the cohomology of cyclic groups.
    
\end{rem}

\begin{lem}\label{Lem2FT}
   Let $F$  be an infinite field  of characteristic  $\neq p$, finitely generated over its prime subfield. Set $\Gamma=\Gal(F^{sep}/F)$. Then, the cup-product pairing $$H^1(\Gamma,\F_p) \times H^1(\Gamma,\F_p) \to H^2(\Gamma,\F_p) $$  has trivial (left) kernel.
\end{lem}

\begin{dem}

  By a limit argument, using restriction/corestriction  for finite extensions of degree prime-to-$p$, one can assume that $\F_p \simeq \mu_p \subset F^\times$, and reduce the question to proving that the cup-product $$H^1(F,\mu_p) \times  H^1(F,\mu_p)  \to \Br(F)$$ has trivial kernel. [This reduction applies to any field $F$.] Pick a non-zero element in $H^1(F,\mu_p)$, corresponding via Kummer theory to $(x) \in F^ \times/(F^ \times)^ p $, where $x$ is not a $p$-th power. We need to find some $(y) \in H^1(F,\mu_p) $ such that $$(x) \cup (y) \neq 0 \in H^2(F,\mu_p).$$ 
Assume that $F$ is a global field, and denote by $V$ the set of all its places. One knows that the group \[\Sha^1(F,\mu_p):= \Ker(H^1(F,\mu_p) \to \prod_{v \in V} H^1(F_v,\mu_p) ) \] vanishes. Let $v \in V$ be such that $x$ is unramified at $v$, and $(x)_v \neq 0 \in H^1(F_v,\mu_p)$. By local class field theory, there exists $y_v \in H^1(F_v,\mu_p)$, such that $$(x)_v \cup (y_v)  \neq 0 \in \Br(F_v).$$ Since the restriction map $H^1(F,\mu_p) \to H^1(F_v,\mu_p)$ is surjective, one may assume that $y_v=y \in F^ \times$. Then indeed,  one has $(x) \cup (y)  \neq 0 \in \Br(F).$ It remains to deal with an $F$ which is not a global field. In that case, denote by $\F \subset F$  the prime subfield. Denote by $d$ the transcendence degree of $F$ over $\F$. If $\F=\F_l$, then $d \geq 2$.  If $\F=\Q$, then $d \geq 1$.  Replacing $x$ by $x t^p$ for some $t \in F^\times$, one may assume that $x$ is transcendental over $\F$. Pick a transcendence basis $x=x_1,x_2,\ldots, x_d$ of $F$ over $\F$. Set $E:=\F(x_2, \ldots, x_d)$.  Denote by $F' \subset F$ the separable closure of $E(x)$ in $F$. The extension $F/F'$ is finite and purely inseparable, hence of degree prime-to-$p.$ Moreover, $x \in F'$. For our purpose, using restriction/corestriction, we may replace $F$ by $F'$, thus reducing to the case $F/E(x)$ separable. Replacing $E$ by its algebraic closure in $F$, one may further assume that $E$ is algebraically closed in $F$. Thus, $F$ is the function field $E(C)$ of a geometrically connected proper  smooth $E$-curve $C$. The extension $F/E(x)$ corresponds to a finite separable $E$-morphism  $C \to \P^1$. Define $ K:=E(x)[T]/(T^p-x)$. The assumption $x \notin (F^ \times)^ p$ translates as: $L:=K \otimes_{E(x)} F$ is a field. It is a finite separable extension of $E(x)$. Since $E$ is an infinite field, finitely generated over its prime subfield, it is Hilbertian. Hence, there exist infinitely many $e \in E$, such that the specialisation of $L/E(x)$ at $x \mapsto e$ is a field extension $L_e /E$. In particular, there is a unique closed point  of $C$ lying above such an $e$.  This way, one produces infinitely many closed points $c \in C$, such that $v_c(x)=0$, and $x(c) \in E(c)^\times$  is not a $p$-th power, where $E(c)/E$ denotes the residue field of $c$. Let $\pi_c \in F^\times$ be a uniformizer at $c$ (in other terms, $v_c(\pi_c)=1$). Then, 
 $$(x) \cup (\pi_c) \neq 0 \in H^2(F,\mu_p),$$ e.g. because its residue at $c$ is non-vanishing, namely $$\Res_c ((x) \cup (\pi_c))=x(c) \neq 0  \in H^1(E(c),\mu_p).$$ 
 This finishes the proof.
\end{dem}

\begin{rem}
    As pointed out by the referee, using similar arguments, Lemma \ref{Lem2FT} also holds for function fields of varieties of dimension at least $2$ over arbitrary fields.
\end{rem}

The following lemma is an analogue of \cite{Se}, I.4.5, Proposition 30:
\begin{lem}\label{LemmaTrivLK}
    Let $\Gamma$ be a Demu$\check{s}$kin pro-$p$-group. Let $V$ be a $(\Gamma,1)$-bundle. Then, the cup-product pairing  $$H^1(\Gamma,V )\times H^1(\Gamma,V^\vee) \to H^2(\Gamma,\F_p)\simeq \F_p $$ has trivial (left) kernel.
\end{lem}

\begin{dem}

Since $\Gamma=\Gamma_p$ is pro-$p$,  there exists an extension of  $(\Gamma,1)$-bundles $$(E): 0 \to W \xrightarrow{i} V \xrightarrow{\pi} \F_p\to 0.$$ We may then prove the result by induction on the dimension.
    The case $\dim(V)=1$ follows from Definition \ref{Defidemu}.  Suppose the Lemma holds for $W$.  Denote the  cohomology class  of $(E)$ by $$ e \in H^1(\Gamma, W ).$$ If $e=0$, then $V=W \bigoplus \F_p$, and the statement readily follows from the induction hypothesis. Henceforth, assume $e\neq 0$. Dualising  $(E)$, one gets the extension of  $(\Gamma,1)$-bundles $$(E^\vee) : 0 \to \F_p \to V^\vee \xrightarrow{q} W^\vee  \to 0.$$

  By induction, one knows that the pairing $$H^1(\Gamma,W )\times H^1(\Gamma,W^\vee) \to H^2(\Gamma,\F_p)\simeq \F_p $$ has trivial left kernel.  We shall thus tacitly  identify  $H^1(\Gamma,W )$ to an $\F_p$-subspace of $H^1(\Gamma,W^\vee)^\vee$. With respect to this pairing, the orthogonal of the one-dimensional space $$\F_p\cdot e \subset H^1(\Gamma,W ) $$ is  the hyperplane
    $$ \Ker(H^1(\Gamma,  W^\vee) \xrightarrow{\delta} H^2(\Gamma, \F_p ) \simeq \F_p), $$ where $\delta$ is the connecting map in cohomology arising from the extension $(E^\vee)$. Indeed, the kernel of the linear form $\delta$ is the image of the map $i^\vee_* : H^1(\Gamma, V^\vee) \to H^1(\Gamma, W^\vee)$, and for all $f \in H^1(\Gamma, V^\vee)$, we have (up to sign) $e \cup i^\vee_*(f) = i(e) \cup f = 0$ since $i(e)=0$. Hence $\ker(\delta)= e^\perp$.\\
    Consider a class $v \in H^1(\Gamma,V)$. Assume that $v \cup x=0 ,$  for all $ x \in  H^1(\Gamma,V^\vee)$. By naturality of the cup-product, and a little diagram chase left to the reader, this implies that $\pi_*(v) \in H^1(\Gamma,\F_p)$ is orthogonal to the whole of $ H^1(\Gamma,\F_p)$. By condition (2) of Definition \ref{Defidemu}, one gets $\pi_*(v)=0$, so that there exists $w \in H^1(\Gamma,W)$, with $i_*(w)=v$. Then, one checks as above that the class  $w$ is orthogonal to the image of$$q_*:H^1(\Gamma,V^\vee ) \to H^1(\Gamma,W^\vee),$$ which is the hyperplane $\Ker(\delta)$. Thus the kernel of the $\F_p$-linear form $(w,.)$ is contained in that of $(e,.)$. It follows that $(w,.)$ is a multiple of $(e,.)$, implying (by induction assumption) that $w$ is zero or collinear to $e$, hence $v=i_*(w)=0$. 
\end{dem}
\begin{lem}\label{LemManaLift}
    Let $\Gamma$ be a $p$-Demu$\check{s}$kin profinite group. Consider a non-split extension of  $(\Gamma,1)$-bundles $$(E): 0 \to L \xrightarrow{i} W \xrightarrow{\pi} V \to 0,$$  where $L$ is invertible. Denote its  cohomology class by $$ e \neq 0 \in \Ext^1_{(\Gamma,1)}(V,L)=H^1(\Gamma, V^\vee \otimes L ).$$ Then, the map  $$H^2(\Gamma,L ) \xrightarrow{i_*} H^2(\Gamma,W) $$ is zero. Equivalently, the connecting arrow  arising from $(E)$,$$H^1(\Gamma, V) \to H^2(\Gamma,L), $$ $$ v^1 \mapsto e \cup v^1$$
 is onto.
 \end{lem}

\begin{dem}

Let us prove the first assertion, which is easily seen to be equivalent to the second one. By restriction/corestriction, and a limit  argument (with respect to a pro-$p$-Sylow subgroup  $\Gamma_p \subset \Gamma$), one sees that $\Res(E)$ is a non-split extension of $(\Gamma_p,1)$-bundles.  Likewise, the question is then reduced to the case where $\Gamma$ is a pro-$p$-group. Then $L \simeq \F_p$, and $H^2(\Gamma,\F_p)$ is one-dimensional. It thus suffices to prove that the connecting arrow $H^1(\Gamma, V) \to H^2(\Gamma,L)$ is non-vanishing. This follows from Lemma \ref{LemmaTrivLK} (applied to $V^\vee$).
\end{dem}

\begin{defi}(The wH90 property).\label{def wH90}\\
    Let $\Gamma$ be a profinite group. Let $\Gamma_p \subset \Gamma$ be a pro-$p$-Sylow. Let  $\Z_p(1)$ be a $\Z_p$-module of rank one, equipped with a $\Gamma$-action, occurring  via a continuous character $\Gamma \xrightarrow{\chi} \Z_p^\times$. Say that the pair $(\Gamma, \Z_p(1))$ satisfies wH90 (for ``weak formal Hilbert 90''), if the reduction map $$H^1(\Gamma_p, (\Z/p^r)(1)) \to H^1(\Gamma_p, \F_p(1))$$ is onto, for every $r \geq 2$.
\end{defi}
\begin{rem}
For $p=2$, take the cyclic group of order $2$, $\Gamma:=\Z/2\Z$. It acts on $\Z_2$ non-trivially, by the sign character. Define this $\Gamma$-module to be $\Z_2(1)$. It is then standard that $(\Gamma, \Z_2(1))$ satisfies wH90.
\end{rem}
\begin{rem}

    In \cite{DCF0}, the pair $(\Gamma, \Z_p(1))$ is said to be $(1,\infty)$-cyclotomic, if surjectivity above holds not only for $\Gamma_p$, but also for every open subgroup $\Gamma' \subset \Gamma$. However, to prove the lifting theorems we have in mind, it suffices here to require surjectivity for $\Gamma_p$. Observe that this is equivalent to surjectivity of  $H^1(\Gamma', (\Z/p^r)(1)) \to H^1(\Gamma', \F_p(1))$, for every open subgroup $\Gamma' \subset \Gamma$, of prime-to-$p$ index.
    \end{rem}

\begin{defi}($p$-manageable profinite group).\label{defpmana}\\
    Let $\Gamma$ be a profinite group.  Let  $\Z_p(1)$ be a $\Z_p$-module of rank one, equipped with a $\Gamma$-action. Say that $\Gamma$ is $p$-manageable with respect to $\Z_p(1)$, if $\Gamma$ is $p$-Demu$\check{s}$kin, and the pair $(\Gamma, \Z_p(1))$ satisfies wH90.
\end{defi}
We end this section with a lemma:
\begin{lem} \label{lem Kummer surj}
Let $(\Gamma, \Z_p(1))$ be a pair satisfying (wH90) and $L$ be a $\Gamma$-module of order $p$. Then the natural morphism $H^1(\Gamma, \W_{r+1}(L)(1)) \to H^1(\Gamma, \W_r(L)(1))$ is surjective.
\end{lem}

\begin{proof}
There exists an isomorphism $L \simeq \Z / p\Z$ as $\Gamma_p$-modules. The morphism $H^1(\Gamma_p, \W_{r+1}(L)(1)) \to H^1(\Gamma_p, \W_r(L)(1))$ identifies to $H^1(\Gamma_p, \Z / p^{r+1} (1)) \to H^1(\Gamma_p, \Z / p^r(1))$. This last map is surjective by induction on $r$ and reduction to the (wH90) property. Therefore, by restriction-corestriction, the cokernel of the morphism $H^1(\Gamma, \W_{r+1}(L)(1)) \to H^1(\Gamma, \W_r(L)(1))$ is both $p$-torsion and prime-to-$p$-torsion, hence it is trivial.
\end{proof}

\section{Common properties shared by (seemingly) unrelated groups}
We begin with gathering classical material about topological fundamental groups of surfaces. For convenience, short proofs are included.
\begin{prop}\label{PropTopo}
 Let $\Gamma$ be the topological fundamental group of a closed connected orientable surface $S_g$, of genus $g \geq 1$. Let $M$ be a finite abelian group, equipped with an action of $\Gamma$. The following holds. \begin{enumerate}
 \item{The abelianisation $\Gamma_{ab}$ is isomorphic to $\Z^{2g}$.}
     \item{If the $\Gamma$-action on $M$ is trivial,  there are natural isomorphisms $$H^i(\Gamma, M) \stackrel \sim \to H^i_{singular}(S,M)$$ and  $$H^2(\Gamma, M) \stackrel \sim \to M.$$} \item{ The cup-product  $$ H^1(\Gamma, M) \times H^1(\Gamma, \Hom(M,\Q/\Z)) \to H^2(\Gamma, \Q/\Z)=\Q/\Z$$ is a perfect pairing of finite abelian groups.} \item{Denote by $\widehat{\Gamma}$ the profinite completion of $\Gamma$. For $i \leq 2$, consider the natural arrow, given by inflation, $$H^i(\widehat{\Gamma}, M) \xrightarrow{\theta^i} H^i(\Gamma, M) \, ,$$ where cohomology used at the source is that of a profinite group, with  discrete coefficients (it is $\varinjlim H^i(\Gamma/\Gamma_0,M)$, where $\Gamma_0 \subset \Gamma$ runs through normal subgroups of finite index, acting trivially on $M$) and cohomology at the target is usual group cohomology. Then, $\theta ^i$ is an isomorphism.}
    
 \end{enumerate}
   
\end{prop}

\begin{dem}
To prove (1), it is convenient to use the classical presentation of $\Gamma$ by generators and relations: \[ \Gamma= \langle X_1, Y_1, \ldots, X_g, Y_g \, \vert \, [X_1,Y_1] \ldots [X_g,Y_g]=1 \rangle.\] From there, the result is obvious.
Let us prove (2). The fact that $H^i(\Gamma, M) \to H^i_{sing}(S,M)$ is an isomorphism, is a general fact that holds because the universal cover of $S$ is contractible. The second assertion is classical, from the assumptions made on $S$. With the help of (1), item (3) is a direct consequence of Poincar\'e duality. 

It remains to deal with (4). It is clear that $\theta^0$ is an isomorphism. Let $\Gamma_0 \subset \Gamma$ be a normal subgroup of finite index, acting trivially on $M$. There is the inflation-restriction sequence \[ 0 \to H^1(\Gamma / \Gamma_0,M) \xrightarrow{\Inf} H^1(\Gamma,M)  \xrightarrow{\Res} H^1(\Gamma_0,M)= \Hom(\Gamma_0,M). \]  It follows from the same exact sequence for any finite index normal subgroup $\Gamma_1$ of $\Gamma$ that $\theta^1$ is injective. Given a class $c \in H^1(\Gamma,M)$, let $\Gamma_1 \subset \Gamma$ be a normal subgroup of finite index, contained in $\Ker(\Res(c)) \subset \Gamma_0$. Using the inflation-restriction sequence for $\Gamma_1$, one sees that $c$ is inflated from  $H^1(\Gamma/\Gamma_1,M)$. This proves surjectivity of $\theta^1$. For $i=2$, and for any $\Gamma_0$ as above, there is an exact sequence \[H^0(\Gamma / \Gamma_0,H^1( \Gamma_0,M) )\xrightarrow{e} H^2(\Gamma / \Gamma_0,M) \xrightarrow{\Inf} H^2(\Gamma,M) . \] Pick $x \in H^2(\Gamma / \Gamma_0,M) ,$ with $\Inf(x)=0$. Pick an invariant class $c \in H^1( \Gamma_0,M)=\Hom(\Gamma_0,M)$, such that $e(c)=x$. As above, let $\Gamma_1 \subset \Gamma$ be any normal subgroup of finite index, contained in $\Ker(\Res(c))$. By a little diagram chase, using the exact sequence above and its analogue for $\Gamma_1$, one concludes that the inflation of $x$ in $H^2(\Gamma / \Gamma_1,M)$ vanishes. This proves injectivity of $\theta^2$. It remains to  prove surjectivity. Pick some $c \in H^2(\Gamma,M)$.

Let us first prove that its  restriction to some subgroup of finite index vanishes. For $\Gamma_0$ as above,  considering  $\Res(c) \in H^2(\Gamma_0,M)$, one first reduces to the case where the action of $\Gamma$ on $M$ is trivial. By d\'evissage on the finite abelian group $M$, one reduces further, to $M=\F_p$. By Poincaré duality, $H^2(\Gamma,\F_p)\simeq \F_p$  has a generator of the shape $x_1 \cup x_2$, for $x_1,x_2 \in H^1(\Gamma,\F_p)$, so that $\Ker(x_1) \subset \Gamma$ does the job. Thus, there exists a normal subgroup of finite index $\Gamma_1 \subset \Gamma$, acting trivially on $M$, and such that $\Res_{\Gamma_1}(c)=0 \in H^2(\Gamma_1,M).$ Introduce the extension of finite $(\Gamma/\Gamma_1)$-modules \[(E): 0 \to M \xrightarrow{\iota} M[\Gamma/\Gamma_1] \to N \to 0,\] where $\iota$ is the natural map. By Shapiro's Lemma, $$\iota_*(c)=0 \in H^2(\Gamma,M[\Gamma/\Gamma_1]) \simeq H^2(\Gamma_1,M).$$ Hence, $c$ is of the form $(E) \cup x,$ for some $x \in H^1(\Gamma,N)$. Let $ \Gamma_2 \subset \Gamma_1$ be a subgroup of finite index, normal in  $ \Gamma$, and such that  $x \in H^1(\Gamma/\Gamma_2,N)$. As the cup-product of two extensions of finite $(\Gamma/\Gamma_2)$-modules, $c$ is then inflated from   $H^2(\Gamma/\Gamma_2,M)$, proving surjectivity of  $\theta^2$.
\end{dem}

\begin{rem} \label{rem completion}
    A representation $\Gamma \to \GL_d(\Z/p^r)$ is the same as a continuous representation  $\widehat\Gamma \to \GL_d(\Z/p^r)$. Thus, the lifting problems considered in this text are the same for  $\Gamma$ and  $\widehat \Gamma$. Point (4) of the preceding Lemma states they also share the same cohomology groups. [Note  that $H^i(\widehat{\Gamma}, M) = H^i(\Gamma, M)=0,$ for $i \geq 3$.] All this is a particular case of ``cohomological goodness''  (in the sense of \cite{Se}, I.2.6, exercise 2) for surface groups, see for instance \cite{GJZ}, Proposition 3.7.
\end{rem}
\begin{prop}\label{Prop2G}
    Let $(\Gamma,\Z_p(1))$ be one of the following. \begin{enumerate}
        \item{The absolute Galois  group $\Gamma=\Gal(F^{sep} /F)$,   of a local field $F$ ($\R$, $\C$, a finite extension of $\Q_\ell$, or a finite extension of  $\F_\ell((t))$, for a prime $\ell$ possibly equal to $p$). Define the module $\Z_p(1)$ to be $\Z_p$ if $\mbox{char}(F)=p$, or the Tate module of roots of unity of $p$-primary order in $\overline F$, if $\mbox{char}(F) \neq p$.} 
        \item{The profinite completion $\Gamma$ of the topological fundamental group $\Gamma_{g,top}$ of a closed connected orientable  surface $S_g$, of genus $g \geq 1$. Take $\Z_p(1)$ to be the trivial module $\Z_p$.}
    \end{enumerate}
        Then the group $\Gamma$ is $p$-manageable, with respect to $\Z_p(1)$.
       
\end{prop}
    \begin{dem}
    Let us deal with item (1). The case $F=\C$ is trivial, and $F=\R$ is an exercise. If $\mbox{char}(F)=p$, then $F$ is of $p$-cohomological dimension $1$ by Artin-Schreier theory (see \cite{Se}, II, Proposition 3). This is a much stronger property than $p$-manageability. In the remaining cases, $\mbox{char}(F) \neq p$. The fact that $\Gamma$ is $p$-manageable is then a straighforward consequence of local class field theory, stating in particular that $H^2(E, \F_p(1)) \simeq \F_p$, for every finite extension $E/F$. Observe that the property wH90  simply follows from Hilbert's Theorem 90 for $\G_m$, namely $H^1(F,F^{sep \times })=0$. Let us deal with item (2). The wH90 property follows from item (1) of Proposition \ref{PropTopo}, which also holds  for every subgroup of finite index of $\Gamma_{g,top}$. By item (2) of this Proposition, one has $ H^2(\Gamma'_{g',top}, \F_p)= \F_p,$ for every  subgroup of finite index $\Gamma'_{g',top} \subset \Gamma_{g,top}$ (since such a subgroup is of the same geometric origin). If this index is prime-to-$p$, using restriction/corestriction, one sees that the restriction $H^2(\Gamma_{g,top}, \F_p) \to  H^2(\Gamma'_{g',top}, \F_p)$ is injective -- hence an isomorphism of one-dimensional $\F_p$-vector spaces. By a straightforward limit argument, condition (1) of  definition \ref{Defidemu} is thus satisfied. Condition (2) holds by item (3) of Proposition \ref{PropTopo}.
    \end{dem}

    \begin{rem}
        In the Proposition above, $H^1(\Gamma, \F_p)$ is finite-dimensional. However, this is not required for our method to work.
    \end{rem}

   \begin{rem}
    Let $E/F$ be an infinite algebraic extension of a local field $F$, of degree divisible by $p^\infty$, as a supernatural number. This means here that $E=\varinjlim_{i \in \N} E_i$, where $F \subset E_i \subset E_{i+1}$ are finite extensions, whose degrees $[E_{i+1}:E_i]$ are all divisible by $p$. Local class field theory identifies  the restriction $$\Br(E_i ) \to \Br(E_{i+1})$$  to $$ \Q/ \Z \xrightarrow{ [E_{i+1}:E_i] \Id} \Q /\Z.$$ Since $$\Br(E)=\varinjlim_{i \in \N} \Br(E_i),$$ One infers that $\Br(E)[p]=0$, so that the $p$-cohomological dimension of $E$ is $\leq 1$.\\ \noindent  [In particular, this applies to  $E=\varinjlim F(\mu_{p^i})$, the cyclotomic $p$-extension of $F$.]\\ Therefore, all lifting problems considered in this text (for $\Gamma_E$) can easily be solved. 
\end{rem}

\section{When $H^1(\Gamma,\F_p)$ is finite}\label{PDG}

    In this paper,  lifting theorems  hold for all $p$-manageable profinite groups. One can then raise the following natural question.\\ \textit{Let $\Gamma$ be a $p$-Demu$\check{s}$kin profinite group. Does there exist a continuous character $\chi: \Gamma \to \Z_p^\times$, with respect to which $\Gamma$ is  $p$-manageable? }

   As the question is stated, the answer is most likely negative. However, it is positive if $\Gamma$ is a topologically finitely generated pro-$p$-group (i.e. if $\Gamma$ is   Demu$\check{\textup{s}}$kin in the classical sense). This result is found in   \cite{Se2}, 9.3. The next Proposition is slightly more general: there,  $\Gamma$ is not required to be a pro-$p$-group. The proof we provide is constructive.
   
   \begin{prop}\label{FGcase}
       Let $\Gamma$ be a profinite group. Let $\F_p(1)$ be a $\Gamma$-module of order $p$, such that $\dim_{\F_p}(H^2(\Gamma,\F_p(1)))=1$, and such that  the  cup-product pairing  $$ H^1(\Gamma,\F_p) \times H^1(\Gamma,\F_p(1)) \to H^2(\Gamma,\F_p(1)) \simeq \F_p$$ is a perfect pairing of finite abelian groups. \\
          Then, there exists a  lift of $\F_p(1)$  to a  $\Z_p$-module of rank one, equipped with a continuous $\Gamma$-action, and denoted by $\Z_p(1)$, such that the natural arrow $$H^1(\Gamma,\Z/p^r(1)) \to H^1(\Gamma, \F_p(1))$$ is surjective for every $r\geq 1$. This $\Z_p(1)$ is unique up to iso.
           

   \end{prop}

   \begin{dem}
  By induction on $r$, assume built a lift  of $\F_p(1)$ to a  $(\Gamma,r)$-bundle $\Z/p^r(1)$,  satisfying the required property. Let us explain how to build $\Z/p^{r+1}(1)$.
To do so, it is very convenient to use the ring scheme of truncated Witt vectors of length two, $\W_2$. Consider the natural surjection of rings \begin{align*} \pi: \W_2(\Z/p^r)&\to \Z/p^r, \\  (x_0,x_1) &\mapsto x_0.\end{align*} Recall that the Teichm\"uller lift  for line bundles exists in full generality, as defined in \cite{DCFLA}, 4.1. Denote by $\W_2(\Z/p^r(1))$ the Teichm\"uller lift of  $\Z/p^r(1)$: it is a free $\W_2(\Z/p^r)$-module of rank one, equipped with a $\Gamma$-action. It lifts $\Z/p^r(1)$  via $\pi$. Consider the reduction sequence of the $\W_2$-line bundle $ \W_2(\Z/p^r (1))$ (see \cite{DCFLA}, 4.1) \[(\mathcal W_r): 0 \to \Frob_*( \Z/p^r(p)) \to \W_2(\Z/p^r (1))\xrightarrow{\pi(1)} \Z/p^r (1)\to 0.\] It is an exact sequence of $(\W_2(\Z/p^r),\Gamma)$-modules. Recall that $\Frob_*( \Z/p^r(p))$ denotes the abelian group $ \Z/p^r(p)$, seen as a $\W_2(\Z/p^r)$-module via the Witt vector Frobenius (a ring homomorphism)  \[\W_2(\Z/p^r) \xrightarrow{\Frob} \Z/p^r,\] \[(x_0,x_1) \mapsto x_0^p+px_1.\] [Likewise, on the right of $(\mathcal W_r)$, $\Z/p^r(1)$ should  be denoted by $\pi_*(\Z/p^r(1))$.]\\
The induced connecting map in cohomology, is a $\W_2(\Z/p^r)$-linear map  \[ \beta: H^1(\Gamma,\Z/p^r(1)) \to  \Frob_*( H^2(\Gamma,\Z/p^r(p))).\]
Observe that the ideal \[V:=\Ver(\Z/p^r) =\Ker(\pi)=\{0 \} \times \Z/p^r \subset \W_2(\Z/p^r) \] annihilates $H^1(\Gamma,\Z/p^r(1))$ (considered as a $\W_2(\Z/p^r)$-module via $\pi$). Hence, $\beta$ takes values in the $V$-torsion sub-module of $\Frob_*( H^2(\Gamma,\Z/p^r(p)))$.   From the equality $\Frob(0,x)=px$, it is straightforward to see that ``$V$-torsion=$p$-torsion'', for the $\W_2(\Z/p^r)$-module  $\Frob_*( H^2(\Gamma,\Z/p^r(p)))$. Applying $H^*(\Gamma,.)$ to the exact sequence \[0 \to\F_p(p) \to  \Z/p^r(p) \to  \Z/p^{r-1}(p) \to 0,\]  one gets a surjection of $\F_p$-vector spaces \[ H^2(\Gamma,\F_p(p)) \stackrel {\nat}  \twoheadrightarrow H^2(\Gamma,\Z/p^r(p))[p].\]
Thus, there exists a linear map of $\F_p$-vector spaces $\overline \beta$, giving the following factorisation of $\beta$: \[ H^1(\Gamma,\Z/p^r(1)) \stackrel {\nat}  \twoheadrightarrow H^1(\Gamma,\F_p(1))  \xrightarrow{\overline \beta } H^2(\Gamma,\F_p(p)) \xrightarrow{\nat} \Frob_*( H^2(\Gamma,\Z/p^r(p))).\] [The surjection $\nat$ on the left arises from   $\Z/p^r(1) \twoheadrightarrow \F_p(1)$.]\\
Since $\F_p(p-1)=\F_p$, by Poincaré duality mod $p$, $\overline \beta$ is of the shape \[\overline \beta: H^1(\Gamma,\F_p(1)) \to  H^2(\Gamma,\F_p(p)) \simeq \F_p,\]  \[x \mapsto x \cup e_r,\] for some (unique)  $e_r \in  H^1(\Gamma,\F_p(p-1))$.
Observe that, as a byproduct of the discussion on the $V$-torsion above, one gets \[\F_p= \F_p(p-1) = \Hom_{\F_p\text{-Mod}} (\F_p(1), \F_p(p))= \Hom_{\W_2(\Z/p^r)\text{-Mod}} (\Z/p^r(1),\Frob_*( \F_p(p))) \] \[=\Hom_{\W_2(\Z/p^r)\text{-Mod}} (\Z/p^r(1),\Frob_*( \Z/p^r(p))).\]
This provides an action of (the group) $\F_p$, on the trivial extension of $(\W_2(\Z/p^r),\Gamma)$-modules \[ 0 \to \Frob_*( \Z/p^r(p)) \to  \ast  \to  \Z/p^r(1) \to 0,\] by automorphisms of extension. Twisting it by the $\F_p$-torsor $e_r$ yields an extension   of $(\W_2(\Z/p^r),\Gamma)$-modules  \[ (\mathcal E_r): 0 \to \Frob_*( \Z/p^r(p)) \to  \ast  \to  \Z/p^r(1) \to 0,\] 
which is geometrically trivial  (=trivial as an extension of $\W_2(\Z/p^r)$-modules). Its connecting homomorphism is $\beta$ as well. Form the Baer difference \[\mathcal D_r:=(\mathcal W_r-\mathcal E_r): 0 \to \Frob_*(\Z/p^r(p)) \to L_r \xrightarrow{\lambda} \Z/p^r(1)\to 0.\]  It is an extension of $(\W_2(\Z/p^r),\Gamma)$-modules. Since $\mathcal E_r$ is geometrically trivial, $\mathcal D_r$ is geometrically isomorphic to $\mathcal W_r$. Consequently, its middle term $L_r$, as a $\W_2(\Z/p^r)$-module, is free of rank one.  By compatibility of connecting maps to Baer sum, one sees that the connecting map of $\mathcal D_r$ vanishes. Equivalently, the induced arrow $$H^1(\Gamma,L_r) \xrightarrow{H^1(\lambda)} H^1(\Gamma,\Z/p^r(1))$$ is surjective. Introduce the homomorphism $\Phi$ of Lemma \ref{LemCar}, and set
 $$\Z/p^{r+1}(1):=L_r\otimes_{\Phi} \Z/p^{r+1}.$$ As a $\Z/p^{r+1}$-module, $\Z/p^{r+1}(1)$ is free of rank one. Observe that $ \Z/p^{r+1}(1) \otimes_{\Z/p^{r+1}} \F_p$ is isomorphic to $\F_p(1)$ (given at the very start). However,  a priori, $$ \Z/p^r(1)':=\Z/p^{r+1}(1) \otimes_{\Z/p^{r+1}} \Z/p^r$$ need not be isomorphic to $\Z/p^r(1)$ (constructed in the previous step). Consider the natural reduction sequence of $(\Gamma,r+1)$-modules \[\mathcal K_{r+1}: 0 \to \Z/p^r(1)' \to \Z/p^{r+1}(1) \to \F_p(1)\to 0.\]Using the commutative diagram of Lemma \ref{LemCar}, one gets a commutative diagram of $(\Gamma,r+1)$-modules \[\xymatrix{ L_r\ar[r]^-{x \mapsto x \otimes 1} \ar[d]^\lambda & \Z/p^{r+1}(1) \ar[d]^{\nat_{r+1}} \\ \Z/p^r (1)\ar[r]^{\nat_r} & \F_p (1).}\] Since $H^1(\lambda)$ and $H^1(\nat_r)$ are surjective,  so is $$H^1(\nat_{r+1}): H^1(\Gamma,\Z/p^{r+1}(1)) \to H^1(\Gamma,\F_p(1)).$$ By uniqueness of $\Z/p^r(1)$, this shows that $\Z/p^{r}(1)'$  is actually isomorphic to $\Z/p^r(1)$. Observe that, as in Lemma \ref{lem Kummer surj}, the arrow $$H^1(\Gamma,\Z/p^{r+1}(1)) \to H^1(\Gamma,\Z/p^r(1))$$ is then surjective as well. \\ Let us prove uniqueness of $\Z/p^{r+1}(1)$. Denote by $\chi_r:\Gamma \to (\Z/p^r)^\times$ the character corresponding to $\Z/p^r(1)$. Assume that $\Z/p^{r+1}(1)'$ is some other choice for a suitable  lift of $\Z / p^r(1)$. Its character is then of the shape \[ \chi'_{r+1}=\chi_{r+1}+ \epsilon,\] where $$\epsilon:\Gamma \to (1+p^r\Z/p^{r+1}\Z)^\times \simeq \Z/p$$ is simply a mod $p$  character, corresponding to an extension of $(\Gamma,r)$-modules  \[ (\epsilon): 0 \to \F_p \to \ast \to\Z/p^r \to 0.\]  Consider the reduction sequences of $(\Gamma,r+1)$-modules \[ 0 \to \F_p(1) \to \Z/p^{r+1}(1) \to \Z/p^r(1)\to 0\] and \[ 0 \to \F_p(1) \to \Z/p^{r+1}(1)' \to \Z/p^r(1)\to 0.\] Since $\Z/p^{r+1}(1)$ and $\Z/p^{r+1}(1)'$ satisfy item (1), the connecting maps $H^1(\Gamma,\Z/p^r(1)) \to H^2(\Gamma,\F_p(1))$ of both extensions vanish.  Also, the  Baer difference of these extensions is  $(\epsilon)(1)$, i.e. $(\epsilon) \otimes_{\Z/p^r}\Z/p^r(1)$. Using compatibility of Baer sum to connecting maps, and surjectivity of $ H^1(\Gamma,\Z/p^r(1)) \to  H^1(\Gamma,\F_p(1))$, a small chase reveals that  $\epsilon \in H^1(\Gamma,\F_p)$ lies in the (left) kernel of the perfect cup-product pairing  $$ H^1(\Gamma,\F_p) \times H^1(\Gamma,\F_p(1)) \to H^2(\Gamma,\F_p(1)). $$ Hence $\epsilon$ vanishes, proving the uniqueness of $\Z / p^r(1)$.  
 \end{dem}

\begin{rem}
Assume that the premises of Proposition \ref{FGcase} also hold for every open subgroup $\Gamma' \subset \Gamma$, in place of $\Gamma$ (for the same coefficients $\F_p(1)$, restricted to $\Gamma'$). Observe that this is true in the Demu$\check{\textup{s}}$kin  case, i.e. when $\Gamma$ is a pro-$p$-group. Then,   for every such $\Gamma'$, $H^1(\Gamma',\Z/p^r(1)) \to H^1(\Gamma', \F_p(1))$ is also surjective. This  means that the pair $(\Gamma,\Z_p(1))$ is $(1,\infty)$-cyclotomic,    in the   sense of \cite{DCF1}.
\end{rem}

To conclude, we deal with the Lemma that was used in the previous proof.
  \begin{lem}\label{LemCar}
   For each $r \geq 1$,  there exists  a  ring homomorphism $$\Phi: \W_2(\Z/p^r) \to \Z/p^{r+1},$$ such that the diagram of ring homomorphisms\[\xymatrix{ \W_2(\Z/p^r) \ar[r]^\Phi \ar[d]^\pi & \Z/p^{r+1} \ar[d]^\nat \\ \Z/p^r \ar[r]^\nat & \F_p }\] commutes, where $\pi(x_0,x_1)=x_0$ is the natural surjection.
\end{lem}

\begin{dem}
 Consider the Witt vector Frobenius  of the ring $\Z/p^{r+1}$, given by the first Witt polynomial:  \[\W_2(\Z/p^{r+1}) \xrightarrow{\Frob} \Z/p^{r+1},\] \[(x_0,x_1) \mapsto x_0^p+px_1.\]
    As in \cite{DCFLA}, Lemma 5.5 (to which we refer for details), one proves that $\Frob$ factors through the natural arrow $\W_2(\Z/p^{r+1}) \to \W_2(\Z/p^r)$, giving rise to the promised ring homomorphism. The commutativity of the diagram is readily checked.
\end{dem}

\begin{exo}
    Show that the ring $\W_2(\Z/p^r)$ is a $(\Z/p^{r+1})$-algebra.
\end{exo}

\section{Gluing, lifting and gluifting  extensions of $(\Gamma,r)$-bundles}\label{Sectiongluing}
 In this section, $\Gamma$ is any group, or profinite group. In the latter case, all representations are assumed to be continuous.
 
   \begin{defi}(Gluing)\\
   Assume given two extensions of $(\Gamma,r)$-bundles,
  \[\mathcal E_{d,r}: 0 \to V_{1, r}   \to V_{d,r}  \to V_{d/1,r}  \to 0 \] and 
   \[\mathcal F_{d,r}: 0 \to  V_{d/1,r} \to W_{d,r} \to  L_{d+1,r} \to 0, \]  where $V_{1, r} $  and $L_{d+1,r}$ are invertible, and $V_{d,r}$, $W_{d,r}$   are $d$-dimensional.\\
   A gluing of  $\mathcal E_{d,r}$ and $\mathcal F_{d,r}$ is a pair $(\mathcal E_{d+1,r},\phi_{r})$, consisting of an extension of $(\Gamma,r)$-bundles \[\mathcal E_{d+1,r}: 0 \to V_{d,r} \to V_{d+1,r} \to    L_{d+1,r} \to 0, \] and an isomorphism of  extensions of $(\Gamma,r)$-bundles   \[\phi_{r}: \pi_*(\mathcal E_{d+1,r}) \stackrel \sim \to  \mathcal F_{d+1,r},\] where $\pi: V_{d,r} \to V_{d/1,r} $ is the natural surjection introduced earlier, and $\pi_*(.)$ denotes the induced push-forward operation, at the level of extensions.\\
   Isomorphisms of gluings are defined in the obvious way.
  \end{defi}
   
  The obstruction to gluing $\mathcal E_{d,r}$ and $\mathcal F_{d,r}$ is the cup-product \[\obs(\mathcal E_{d+1,r}) :=\mathcal E_{d,r} \cup \mathcal F_{d,r}:  0 \to V_{1, r}  \to  V_{d,r} \to  W_{d,r}  \to   L_{d+1,r}\to 0.\] It is a $2$-extension of $(\Gamma,r)$-bundles, whose Yoneda class in $$ \Ext^2_{(\Gamma,r)-Mod}( L_{d+1,r},V_{1,r} )=H^2(\Gamma,L_{d+1,r}^{\vee} \otimes  V_{1,r} )$$ vanishes if, and only if, a pair $(\mathcal E_{d+1,r},\phi_{r})$ as above exists.
  
 \begin{defi}(Lifting).\\
Assume given an extension of $(\Gamma,r)$-bundles
  \[\mathcal E_r: 0 \to V_{j, r}   \to V_{k,r}  \to V_{k/j,r}  \to 0. \]  
   A lifting of  $\mathcal E_r$ (mod $p^{r+1}$)  is a pair $(\mathcal E_{r+1},\psi_r)$,   where\[\mathcal E_{r+1}: 0 \to V_{j,r+1} \to V_{k,r+1} \to    V_{k/j,r+1} \to 0 \] is an extension of $(\Gamma,r+1)$-bundles,   and where    \[\psi_r: q(\mathcal E_{r+1}) \stackrel \sim \to  \mathcal E_r \] is an isomorphism of  extensions of $(\Gamma,r)$-bundles.\\\noindent  [Recall the notation $q(.)=(.)\otimes_{\Z /p^{r+1}}(\Z/ p^r)$.] \\
   Isomorphisms of liftings are defined in the obvious way.
 \end{defi}

\begin{defi}
    Define  $$\GL(V_{j,r} \subset V_{k,r}) \subset \GL(V_{k,r}), $$ resp. $$\End(V_{j,1} \subset V_{k,1})  \subset \End(V_{k,1}),$$ as the subgroup,  resp. $\F_p$-subspace, of automorphisms, resp. endomorphisms, preserving $V_{j,r}$, resp. $V_{j,1}$. In block form, it is given by
 \[ \begin{pmatrix}
 \ast  & \ast \\ 0 & \ast 
 \end{pmatrix} \subset  \begin{pmatrix}
 \ast  & \ast \\ \ast & \ast

 \end{pmatrix},\] where blocks are of sizes $j$ and $k-j$. There is a natural reduction sequence $$ 0 \to  \End(V_{j,1} \subset V_{k,1}) \xrightarrow i  \GL(V_{j,r+1} \subset V_{k,r+1}) \xrightarrow q \GL(V_{j,r} \subset V_{k,r}) \to 1,$$ where $q$ is the natural reduction, and $$ i (\epsilon)= \Id+p^r \epsilon,$$ for $\epsilon \in \End(V_{j,1} \subset V_{k,1})$.

\end{defi}

Let $\cE_r$ be an extension of $(\Gamma,r)$-bundles. 
Via a classical construction in non-abelian cohomology (see \cite{Se} chapter 1, 5.6), the obstruction to lifting  $\mathcal E_r$ mod $p^{r+1}$  is a natural class  $$\obs(\mathcal E_{r+1}) \in H^2(\Gamma, \End(V_{j,1} \subset V_{k,1}) ).$$  
    \begin{defi}(Gluifting)\label{defigluift}\\
  Assume given two extensions of $(\Gamma,r+1)$-bundles,
  \[\mathcal E_{d,r+1}: 0 \to V_{1, r+1}   \to V_{d,r+1}  \to V_{d/1,r+1}  \to 0 \] and 
   \[\mathcal F_{d,r+1}: 0 \to  V_{d/1,r+1} \to W_{d,r+1} \to  L_{d+1,r+1} \to 0, \]
   and  a gluing
   $$(\mathcal E_{d+1,r},\phi_r)$$  of  the extensions of $(\Gamma,r)$-bundles $\mathcal E_{d,r}:=q(\mathcal E_{d,r+1})$ and $\mathcal F_{d,r}:=q(\mathcal F_{d,r+1})$.\\
 A lifting of the gluing $(\mathcal E_{d+1,r},\phi_r)$, is the data of a gluing $(\mathcal E_{d+1,r+1},\phi_{r+1})$ of  $\mathcal E_{d,r+1}$ and $\mathcal F_{d,r+1}$, together with an isomorphism \[\theta_r: q(\mathcal E_{d+1,r+1},\phi_{r+1}) \stackrel \sim \to (\mathcal E_{d+1,r},\phi_r),  \] as gluings of $\mathcal E_{d,r}$ and $\mathcal F_{d,r}$.\\
 Altogether, the data of $(\mathcal E_{d+1,r+1},\phi_{r+1}, \theta_r)$ is called a gluifting of $(\mathcal E_{d,r+1},\mathcal F_{d,r+1},\mathcal E_{d+1,r},\phi_r) $.
   \end{defi}
   
Gluifting is related to Grothendieck's ``extensions panach\'ees'', and makes sense in a much more general stacky setting,  worth investigation. However, we stick here to an elementary concrete approach.
Let us now explain how the obstruction to gluifting is  a natural reduction of $\obs(\mathcal E_{d+1,r+1})$ to a mod $p$ cohomology class. 
\begin{lem} \label{lem gluifting obs}
    The obstruction to gluifting, as above, is a natural class  $$c_1 \in \Ext^2_{(\Gamma,1)}( L_{d+1,1},L_{1,1})=H^2(\Gamma,L_{d+1,1}^{\vee} \otimes_{\F_p}  L_{1,1}),$$ such that  $$i_*(c_1)=\obs(\mathcal E_{d+1,r+1}) \in H^2(\Gamma,L_{d+1,r+1}^{\vee} \otimes_{\Z}  L_{1,r+1}).$$
\end{lem}

\begin{dem}

 To simplify, assume first that $L_{d+1,r+1}=\Z/p^{r+1}$ is trivial.\\ 
 Consider the natural surjection of $(\Gamma,r+1)$-modules 
 \[  V_{d,r} \oplus  (V_{d,r+1}/L_{1,r+1}) \stackrel {f} \to V_{d/1,r} \to 0, \] given by $$f(v,l):= \pi(v)-q(l).$$ It fits into an extension  of $(\Gamma,r+1)$-modules \[\mathcal Q: 0 \to N  \to  V_{d,r} \oplus  V_{d/1,r+1} \stackrel {f}\to V_{d/1,r} \to 0, \]whose  kernel $N$ naturally fits into the exact sequence
 \[\mathcal N: 0 \to L_{1,1}  \to    V_{d,r+1} \stackrel {s} \to N \to 0,\] where \[ s(v):= (q(v),\pi( v)) \in N \subset  V_{d,r} \oplus  V_{d/1,r+1}.\]  The injection in $\mathcal N$ is given by the composite inclusion $$ L_{1,1} \stackrel {\iota} \hookrightarrow V_{d,1}\stackrel {i}   \hookrightarrow V_{d,r+1}.$$

Since $L_{d+1,r+1}=\Z/p^{r+1}$, we can  consider   $\mathcal E_{d+1,r}$, resp.  $\mathcal F_{d,r+1}$, as a torsor under the $(\Gamma,r+1)$-module $V_{d,r}$, resp. $ (V_{d,r+1}/L_{d,r+1})$. Taking their direct product, one gets a torsor $X$, under  the  $(\Gamma,r+1)$-module $V_{d,r} \oplus  (V_{d,r+1}/L_{1,r+1})$.
The data of the gluing $(\mathcal E_{d+1,r},\phi_r)$  then yields a natural trivialization of  $q_*(X)$, which is a $\Gamma $-torsor under $V_{d/1,r}$. Using the extension $\mathcal Q$, we get a natural torsor $Y$, under $N$, together with a natural isomorphism  $X \stackrel \sim \to \iota_*(Y)$. Lifting the gluing  $(\mathcal E_{d+1,r},\phi_1)$ is then equivalent to  lifting $Y$, to a $\Gamma$-torsor under $V_{d,r+1}$, via $s$.  Using the connecting map associated to $\mathcal N$, one sees this is  obstructed by a class $$c_1 \in H^2(\Gamma,L_{1,1} ).$$ It is left to the reader, to check that  $i_*(c_1)=\obs(\mathcal E_{d+1,r+1}).$ \\  
The general case, where $L_{d+1,r+1}$ is not assumed to be trivial, reduces to the previous one by  applying $(.\otimes L_{d+1,r+1} ^{\vee})$, the tensor product of $(\Gamma,r+1)$-modules. 
In other words, replacing all $(\Gamma,r+1)$-modules $M$ by $M \otimes  L_{d+1,r+1} ^{\vee}$, we are sent back to the case $ L_{d+1,r+1}= \Z/ p^{r+1}$. The proof is complete.
\end{dem}

\section{Lifting (wound) Kummer flags}

In this section, $\Gamma$ is a $p$-manageable profinite group, with respect to a given $\Z_p(1)$. \\ For instance, one may take one of the two pairs of Proposition \ref{Prop2G}. More generally, by Proposition \ref{FGcase}, lifting theorems of this section apply to all profinite groups satisfying mod $p$ Poincaré duality in dimension $2$.

 Let $\nabla_r = (V_{d,r})$ be a $d$-dimensional complete $(\Gamma,r)$-flag. Under suitable assumptions, we prove  that $\nabla_r$ lifts to a complete $(\Gamma, r+1)$-flag $\nabla_{r+1}$ -- in a very strong ``step-by-step'' sense.





\subsection{Wound Kummer flags}
\begin{defi}(wound flag) \label{defiwound}\\
    A complete $(\Gamma,r)$-flag $\nabla = (V_{i,r})$ is said to be \emph{wound} (\emph{ployé} in French), if for all $1 \leq i \leq d-1$, the extension of $(\Gamma,1)$-modules
\[0 \to L_{i,1} \to P_{i,1} := V_{i+1,1}/V_{i-1,1} \to L_{i+1,1} \to 0\]
does not split.
\end{defi}
\begin{rem}
The flag $\nabla$ is wound if and only if its mod $p$ reduction is wound, as a complete $(\Gamma,1)$-flag.
\end{rem}

\begin{lem}\label{remwoundunique}

    A   complete $(\Gamma,1)$-flag $\nabla = (V_{i,1})$ is  wound, if and only if it is the only complete flag with which the $(\Gamma,1)$-bundle $V_{d,1}$ can be equipped. 
\end{lem}

\begin{dem}
If $\nabla$ is not wound, there exists an $i$ such that the extension of Definition \ref{defiwound} splits. Exchanging  factors of the grading $ V_{i+1,1}/V_{i-1,1} \simeq  L_{i,1} \oplus L_{i+1,1} $ then gives another complete flag on $V_{d,1}$. Conversely, assume that $\nabla$ is wound. Let us prove, by induction on $d$, that $L_{1,1}$ is the only fixed point of the natural $\Gamma$-action on the  finite projective space (of $\F_p$-lines in $V_{d,1}$) $\P(V_{d,1})$.  Let $L \subset V_{d,1}$ be a $\Gamma$-invariant line. The quotient flag $\nabla/V_{1,1}$ (on $V_{d,1}/V_{1,1}$)  is wound, so that by induction one gets $L \subset V_{2,1}.$ If $L \neq V_{1,1}$, then $L$ would provide a splitting of the surjection $V_{2,1} \to L_{2,1}$, contradicting the fact that $\nabla$ is wound. Therefore, $L= V_{1,1}$ and the induction step is complete. Uniqueness of $\nabla$ then follows, again by induction on $d$. Indeed, let $\nabla'=(V_{i,1}')$ be another complete flag on $V_{d,1}$. Then $V_{1,1}'=V_{1,1}$, by what precedes. By induction, one then gets $\nabla/V_{1,1}=\nabla'/V_{1,1}$, so that $\nabla=\nabla'$, as desired.
\end{dem}
\begin{defi}(wound Kummer flag)\label{defwoundk} \\
Let $\nabla_{d,r}$ be a wound complete $(\Gamma,r)$-flag. We say that $\nabla_{d,r}$ is a wound Kummer flag if there exists $N \in \Z$, such that, for all $i=1,\ldots,d$, 
\begin{equation}\label{woundeq} L_{i, r} \simeq \W_r(L_{i,1}(i))(N-i). \end{equation}
    
\end{defi}
\begin{rem}
    Assume that $r=1$. Taking $N=0$ in definition above, one sees that $\W_1(L_{i,1}(i))(-i)=L_{i,1}(i)(-i)=L_{i,1}$. \\Hence, every wound  $(\Gamma,1)$-flag is wound Kummer.
\end{rem}
\begin{rem}
    The integer $N$ in the definition above is secondary. Its purpose is to make the notion of a Kummer flag invariant  by  ``global'' (in the sense of independent of $i$) cyclotomic twists, and by dualizing (Exercise below). 
\end{rem}
\begin{exo}
    Let $\nabla_{d,r}=(V_{i,r})$ be a wound Kummer $(\Gamma,r)$-flag. Show that  its dual flag $\nabla_{d,r}^\vee=(V_{d+1-i,r}^\vee)$ is a wound Kummer $(\Gamma,r)$-flag, as well.
\end{exo}
\begin{rem}
The last line of Definition above essentially says that $L_{i,r}=L_{i+1,r}(1)$, after restriction to a subgroup of $\Gamma$ of prime-to-$p$ index and up to a global "cyclotomic" twist. Here are details. Denote by $\chi_L: \Gamma \to (\Z/p^r)^\times$ the character associated to   a one-dimensional $(\Gamma,r)$-bundle $L$. Recall that $ (\Z/p^r)^\times = \F_p^\times \times U_r$, where $U_r=(1+p\Z/p^r \Z)$. [Note that $U_r= \Z/2 \times (\Z/2^{r-2}) $ if $p=2$, and $U_r= (\Z/p^{r-1}) $ if $p > 2$.] Denote by  $$\chi'_L: \Gamma \to U_r$$ the second component of $\chi_L$, with respect to this decomposition. \\Then $\chi'_L=1$, if and only if $L=\W_r(L/p)$ (i.e. $L$ is the Teichm\"uller lift of its mod $p$ reduction). We thus see that $L_{i, r} \simeq \W_r(L_{i,1}(i))(-i)$, if and only if $\chi'_{L_{i,r}} = \chi'_{\Z / p^r (-i)}$. Equivalently, there exists a finite extension $E/F$, of degree prime to $p$, such that  $L_{i, r}\simeq \Z/p^r(-i)$, as $(\Gamma_E,r)$-bundles.\\

\end{rem}

\begin{ex}
Denote by $$\chi(1): \Gamma \to \Z_p^\times$$  the character associated to $\Z_p(1)$. As a triangular representation, a $3$-dimensional wound Kummer $(\Gamma,r)$-flag reads as
\[ \Gamma \xrightarrow{\rho} \begin{pmatrix}
 \chi(-2)\cdot\W_r(\ovl\chi(2)\vareps_2) & \alpha_1 & \alpha_3 \\ 0 & \chi(-1)\cdot\W_r(\ovl\chi(1)\vareps_1) & \alpha_2 \\ 0 & 0 & \W_r(\vareps_0)
 \end{pmatrix} \in \mathbf B_3(\Z/p^r) \]
for  characters $ \vareps_i\colon\Gamma\to\F_p^\times$,  and suitable functions $\alpha_i : \Gamma \to \Z/p^r$, such that both induced mod $p$ representations   \[ \Gamma \to \begin{pmatrix}
\vareps_2 &\ovl  \alpha_1  \\ 0 & \vareps_1 
 \end{pmatrix}  \in \mathbf B_2(\F_p) \] and \[\Gamma \to \begin{pmatrix}
\vareps_1  & \ovl \alpha_2  \\ 0 & \vareps_0
 \end{pmatrix}\in \mathbf B_2(\F_p)  \]  correspond to non-split line bundle extensions. In case $\vareps_i=1$ for $i=0,1,2$, this means that  $\ovl  \alpha_1 $ and $\ovl  \alpha_2 $ are non-zero characters $\Gamma \to (\F_p,+)$.
\end{ex}

The next Proposition states that wound  Kummer flags can be lifted, in a very strong sense. This means they can be lifted step-by-step, regarding both torsion (i.e. lifting from mod $p^r$ to mod $p^{r+1}$) and dimension (i.e. extending a lifting of a truncation of the flag, to a lifting of the whole flag).
\begin{prop}(step-by-step liftability of wound Kummer flags) \label{propploye}\\
Let $\nabla_{d,r}$ be a wound  Kummer $(\Gamma,r)$-flag.\\
 Assume given a wound Kummer flag $\nabla^{\flat}_{d-1,r+1}=(V^{\flat}_{i,r+1})_{1 \leq i \leq d-1}$, together with an isomorphism $\nabla_{d-1,r} \stackrel \phi \cong \nabla^{\flat}_{d-1,r}$. Then, there exists a lift  of $\nabla_{d,r}$, to a wound Kummer flag $\nabla_{d,r+1}$, and an isomorphism $\nabla_{d-1,r+1} \cong \nabla^{\flat}_{d-1,r+1}$, whose reduction mod $p^r$ equals $\phi$.

\end{prop}
Before giving the proof, here is a corollary,  straightforward by induction on $d$.
\begin{coro}\label{CoroLiftWound}
Assume that $\nabla_{d,r}$ is a wound Kummer $(\Gamma,r)$-flag. Then it admits a lift to a wound Kummer $(\Gamma, r+1)$-flag.
\end{coro}

\begin{proof}

We prove the Proposition, by induction on $d = \dim(\nabla_{d,r})$.
If $d=1$, there is nothing to prove, since $\W_{r+1}(L_{1,1})$ is a lift of $\W_r(L_{1,1})$.  \\
If $d=2$, then $\mathcal{V}_{1,r+1}$ consists of the single piece $L_{1,r+1}=\W_{r+1}(L_{1,1}(1))(-1)$. Our job is to lift $[V_{2,r}] \in \Ext^1_\Gamma(L_{2,r}, L_{1,r})$ to a class in $ \Ext^1_\Gamma(L_{2,r+1}, L_{1,r+1})$. Note that, for all $s\geq 0$, $$\Ext^1_\Gamma(L_{2,s}, L_{1, s})=H^1(\Gamma,(L_{2,s}^\vee \otimes L_{1, s})=H^1(\Gamma,\W_s(L_{2,1}^\vee \otimes L_{1,1})(1)).$$ The map $\Ext^1_\Gamma(L_{2,r+1}, L_{1, r+1}) \to \Ext^1_\Gamma(L_{2,r}, L_{1, r})$ can thus be identified to $$H^1(\Gamma, \W_{r+1}(L_{2,1}^\vee \otimes L_{1,1})(1)) \to H^1(\Gamma, \W_{r}(L_{2,1}^\vee \otimes L_{1,1})(1)),$$ which is surjective by Lemma \ref{lem Kummer surj} (consequence of the (wH90) property), concluding the proof.\\
   It remains to treat the case $d > 2$, assuming that the Proposition holds for all flags of dimension $<d$. We introduce the $(d-1)$-dimensional wound Kummer flag $$\nabla_{d/1,r}:= \nabla_{d,r}/L_{1,r}: 0 \subset V_{2/1,r} \subset  V_{3/1,r} \subset \ldots \subset  V_{d/1,r}.$$ \\ Similarly, we introduce the $(d-2)$-dimensional wound Kummer flag $$\nabla^{\flat}_{d-1/1,r+1}:= \nabla^{\flat}_{d-1,r+1}/L_{1,r+1}: 0 \subset V^\flat_{2/1,r+1} \subset  V^{\flat}_{3/1,r+1} \subset \ldots \subset  V^{\flat}_{d-1/1,r+1} \, .$$
   By the induction hypothesis, there exists a wound Kummer lift $\nabla^{\sharp}_{d-1,r+1}$  of $\nabla_{d/1, r}$, compatible with $\nabla^{\flat}_{d-1/1, r+1}$. The obstruction to gluift $\nabla^{\flat}_{d/1,r+1}$ and $\nabla^{\sharp}_{d-1,r+1}$, along  $\nabla_{d,r}$, to a flag $\nabla_{d,r+1}$, is a class $$c_1 \in H^2(\Gamma, L_{d,1}^\vee \otimes L_{1,1})$$ (see Definition \ref{defigluift}, and the discussion thereafter). If $c_1=0$, then  gluifting can be done, and $\nabla_{d,r+1}$ is automatically a wound Kummer flag, compatible with  $\nabla^{\flat}_{d-1,r+1}$ (i.e. with $\phi$). 
    
   Assume that $c_1 \neq 0$. We are going to modify (=adjust) $\nabla^{\sharp}_{d-1,r+1}$, so that gluifting becomes possible. 
    Since $\nabla_{d,r}$ is wound, the extension 
    \[0 \to L_{1, 1} \to V_{2,1} \to L_{2, 1} \to 0\]
    does not split, nor its twist by $L_{d,1}^\vee$. Denote its class by $$p_1 \in \Ext^1_\Gamma(L_{2,1}, L_{1, 1}) \simeq H^1(\Gamma, L_{2,1}^\vee \otimes L_{1, 1}).$$ Since $p_1\neq 0$, Lemma \ref{LemManaLift} implies that there exists a class $$\epsilon^{\sharp} \in \Ext^1_\Gamma(L_{d,1}, L_{2, 1}) = H^1(\Gamma, L_{d,1}^\vee \otimes L_{2,1}),$$ such that $p_1 \cup\epsilon^{\sharp} = c_1$. 
       
    There are natural $\Gamma$-equivariant injections $$\iota: L_{d,1}^\vee \otimes L_{2, 1} \hookrightarrow \End(\nabla^{\sharp }_{d-1,1}) \hookrightarrow \Aut(\nabla^{\sharp }_{d-1,r+1}) ,$$ where the second one is given by the formula $$ f \mapsto  \Id+p^r f.$$ It is a disguise of the  embedding of triangular  subgroups of $\GL_{d-1}(\Z/p^{r+1})$, \[ \begin{pmatrix}
 1 & 0 & 0  & p^r \ast  \\ 0 & 1 & 0 & 0 \\ 0 & 0 &1 & 0 \\ 0 & 0 & 0 & 1

 \end{pmatrix} \subset  \begin{pmatrix}
  \times  & \ast  & \ast & \ast  \\  0 &  \times  & \ast & \ast   \\ 0 & 0 &  \times & \ast \\0 & 0 &0 &  \times
 \end{pmatrix}.\]
 Since $\iota( L_{d,1}^\vee \otimes L_{2, 1})$ is \textit{central} in $\Aut(\nabla^{\sharp }_{d-1,r+1})$,  we can define 
the flag $$\nabla^{\sharp \mathrm{ad} }_{d-1,r+1}:= \nabla^{\sharp}_{d-1,r+1}- \iota_*(\epsilon^{\sharp}).$$ It is a new lift of $\nabla^{\sharp}_{d-1,r}$, arising as a very small (actually as small as possible) deformation of $\nabla^{\sharp}_{d-1,r+1}$. Denote by  $$c_1^{\mathrm{ad}} \in H^2(\Gamma, L_{d,1}^\vee \otimes L_{1,1})$$   the obstruction to glue $\nabla^{\flat}_{d-1,r+1}$ and $\nabla^{\sharp \mathrm{ad}}_{d-1,r+1}$, to a lift $\nabla_{d,r+1}$ of $\nabla_{d,r}$. Using naturality of cup-product, one computes: $$c_1^{\mathrm{ad}}=c_1-p_1 \cup \epsilon^{\sharp}=0. $$
    Therefore, $\nabla^{\sharp \mathrm{ad} }_{d-1,r+1}$ and $\nabla^{\flat}_{d/1,r+1}$ glue, to the desired $\nabla_{d,r+1}$.
\end{proof}

\subsection{Generalisation to a larger class of coefficients}\label{SectGen}
Let $k$ be a (not necessarily perfect) field of characteristic $p$. Then, Proposition \ref{propploye}, as well as all results so far, hold for representations over  the ring of Witt vectors $\W_r(k)$, in place of $\Z/p^r(=\W_r(\F_p))$. Indeed,  $\F_p$-linearity can everywhere be  upgraded to $k$-linearity. For instance, Lemma \ref{LemManaLift} generalizes to finite-dimensional representations of $\Gamma$ over $k$ (in place of $\F_p$). The proof is the same.

\subsection{Kummer flags, in the presence of enough  roots of unity}

In this section, we assume that $\Z/p^{r+1}(1)\simeq \Z/p^{r+1}$. This assumption is satisfied in the topological case. In the arithmetic case where $\Gamma = \Gamma_F$ for a local field $F$, it is equivalent to assuming that $F$ contains a primitive $p^{r+1}$-th root of unity.

\begin{defi}
Let $\nabla_{d,r}$ be a complete $(\Gamma,r)$-flag.

For all $1 \leq k \leq d$, define $i_r(k)$ to be the smallest integer $0 \leq i \leq k-1$ such that the extension of $(\Gamma,r)$-modules
\[0 \to V_{k-1,r}/V_{i,r} \to V_{k,r} / V_{i,r} \to L_{k,r} \to 0\]
splits.
\end{defi}

To state a general lifting theorem, we need the following notion of Kummer flag, defined by induction. Recall that to any flag $\nabla_{d,r}$ of rank $d$, we can attach two flags: its truncation $\nabla_{d-1,r}$, and its quotient  $\nabla_{d/1,r}:= \nabla_{d,r}/V_{1,r}$.
\begin{defi} \label{def kummer}
A complete $(\Gamma,r)$-flag $\nabla_{d,r}$ is a Kummer flag if the following conditions hold.
\begin{enumerate}
    \item{For all $1 \leq k \leq d$, $i_r(k)=i_1(k)$. } 
    \item{For all $k=1, \ldots, d$, $L_{k,r} = \Z/p^r$. In other words, all one-dimensional graded pieces  are  trivial.}
    \item{$\nabla_{d-1,r}$ and $\nabla_{d/1,r}$ are Kummer.}
    \item For all $2 \leq k \leq d$, if $i_r(k) = 0$, then for any splitting $s : L_{k,r} \to V_{k,r}$, the flag $\nabla^{q,s}_{d-1,r} := \nabla_{d,r}/s(L_{k,r})$ is Kummer.

\end{enumerate}
\end{defi}

Several remarks are in order, to illustrate this definition.

\begin{rem} \label{rem kummer general}
 Consider the condition:
  \begin{enumerate}[label=(\roman*)]
   \setcounter{enumi}{1}
      \item { For all $k=1, \ldots, d$, $L_{k,r} =\W_r(L_{k,1})$.}
  \end{enumerate}

    It is less restrictive than Condition (2) above. The main results, Theorem \ref{ThmLiftK} and its corollary, remain valid if (2) is replaced by (ii), in Definition  \ref{def kummer}. Since  this  makes no significant difference,  we have chosen to work with (2).
\end{rem}
\begin{rem}
If $d=1$, a flag $\nabla_{1,r}(=L_{1,r})$ is Kummer if and only if it is attached to the trivial character $\Gamma\to(\Z/p^r)^\times$. \\ If $d=2$, a flag $\nabla_{2,r}$ is Kummer if and only if it is attached to an extension of the trivial character by itself that is either split, or already non-split  modulo $p$.
\end{rem}

\begin{rem}\label{RemKum1}
If $r=1$, then $\nabla_{d,1}$ is a Kummer flag if and only if all of its graded pieces are trivial.
\end{rem}

\begin{rem}
    Condition (1) in Definition \ref{def kummer} is equivalent to:

     \begin{enumerate}[label=(\roman*)]

      \item {  Consider an extension of  $(\Gamma,r)$-modules of the shape
\[(E_r): 0 \to V_{j/i,r} \to V_{k/i,r} \to V_{k/j,r} \to 0,\] for some integers $0 \leq i \leq j \leq k \leq d.$ If $(E_1)$ splits, then  $(E_r)$ splits.}
  \end{enumerate}
Actually, in combination with conditions (3) and (4),  it would suffice to demand (i) for $j=i+1=k-1$. [Thus, as stated, the formulation of the Definition is slightly redundant. It is nonetheless convenient in practice.] \\
Checking these facts is left to the interested reader.
\end{rem}

\begin{rem}
A wound Kummer $(\Gamma,r)$-flag (in the sense of Definition \ref{defwoundk})
 is a Kummer $(\Gamma,r)$-flag  (in the sense of Definition \ref{def kummer}), if and only if $L_{i,1}$ is trivial for every $i=1,\ldots,d$. Indeed, the wound condition implies that $i_r(k)=i_1(k)=k-1$ for every $k=1,\ldots,d$, and the fact that $L_{i,1}$ is trivial and $\Z/p^r\cong\Z/p^r(1)$ as $(\Gamma,r)$-modules implies that $L_{i,r}$ is trivial for every $i=1,\ldots,d$. \\ On the other hand, a Kummer $(\Gamma,r)$-flag is wound if and only if $i_r(k)=i_1(k)=k-1$ for every $k=1,\ldots,d$. One can easily construct an example of a Kummer flag not satisfying this condition, so that not all Kummer flags are wound.
\end{rem}

The following two lemmas will be useful in the proof of the main Theorem.

\begin{lem} \label{lem 1 glue Kummer}
    Let $\nabla_{d,r+1}$ be a complete $(\Gamma,r+1)$-flag. Let $2 \leq k \leq d-1$ and $s : L_{k,r+1} \to V_{k,r+1}$ be a section of $V_{k,r+1} \to L_{k,r+1}$.

    If $\nabla_{d,r}$, $\nabla_{d/1,r+1}$ and $\nabla_{d,r+1}/s(L_{k,r+1})$ are Kummer, then $\nabla_{d,r+1}$ is Kummer.
\end{lem}

\begin{proof}
    The only non-obvious case to check condition (1) in the definition of Kummer flag for $\nabla_{d,r+1}$ is when $k=d$ and $i_1(d)=i_r(d)$ (since $\nabla_{d,r}$ is Kummer) is $0$ or $1$. Then $i_{r+1}(d) = 0$ or $1$ since $\nabla_{d/1,r+1}$ is Kummer. If $i_1(d)=i_r(d)$ is $1$, then $i_{r+1}(d) = 1$. If $i_1(d)=i_r(d)=0$, then $i_{r+1}(d) = 0$ since $\nabla_{d,r+1}/s(L_{k,r+1})$ is Kummer.

    Condition (2) in the definition is obvious.

    Condition (3) is obvious for $\nabla_{d/1,r+1}$ and follows by induction on the dimension for $\nabla_{d-1,r+1}$.

    Condition (4) holds by induction on $d$.
\end{proof}

\begin{lem} \label{lem 2 glue Kummer}
   Let $\nabla_{d,r+1}$ be a complete $(\Gamma,r+1)$-flag such that $i_1(k) \geq 1$ for all $2 \leq k \leq d-1$ and $i_1(d) \geq 2$.
   
   If $\nabla_{d/1,r+1}$ and $\nabla_{d-1,r+1}$ are Kummer, then $\nabla_{d,r+1}$ is Kummer.
\end{lem}

\begin{proof}
    Condition (1) is obvious for all $k \leq d-1$ since $\nabla_{d-1,r+1}$ is Kummer, and for $k=n$ since $i_1(d) \geq 2$ and $\nabla_{d/1,r+1}$ is Kummer.

    Conditions (2) and (3) are obvious.

    Condition (4) is clear because of the assumption on the $i_1(k)$'s.
\end{proof}

We can now state and prove the main result of this section.

\begin{thm}\label{ThmLiftK}
Assume that $\Z/p^{r+1}(1)\simeq \Z/p^{r+1}$.

Let $\nabla_{d,r}$ be a complete Kummer $(\Gamma,r)$-flag.
\begin{itemize}[leftmargin=*]
    \item Consider a Kummer flag $\nabla^\sharp_{d-1,r+1}$, together with an isomorphism $\nabla_{d/1,r} \stackrel {\phi_r} \cong \nabla^\sharp_{d-1,r}$. Then, there exists a Kummer lift $\nabla_{d,r+1}$ of $\nabla_{d,r}$, together with an isomorphism $\nabla_{d/1,r+1}  \stackrel {\phi_{r+1}} \cong \nabla^\sharp_{d-1,r+1}$, which lifts $\phi_r$.
    \item Consider a Kummer flag $\nabla^\flat_{d-1,r+1}$, together with an isomorphism $\nabla_{d-1,r} \stackrel {\phi_r} \cong \nabla^\flat_{d-1,r}$. Then, there exists a Kummer lift $\nabla_{d,r+1}$ of $\nabla_{d,r}$, together with an isomorphism $\nabla_{d-1,r+1}  \stackrel {\phi_{r+1}} \cong \nabla^\flat_{d-1,r+1}$, which lifts $\phi_r$.
\end{itemize}

In particular, $\nabla_{d,r}$ lifts to a complete Kummer $(\Gamma, r+1)$-flag $\nabla_{d,r+1}$. 
\end{thm}

Before  the proof, we give a straightforward corollary. Let  $$\nabla_{d,r}=V_{1,r} \subset V_{2,r} \subset \ldots \subset V_{d,r}$$ be a  Kummer flag of $(\Gamma,r)$-bundles, for some $r \geq 2$. Then, every  $c \in H^1(\Gamma, V_{d,1})$ can be seen as the class of an extension of $(\Gamma,1)$-bundles $$0 \to V_{d,1} \to V_{d+1,1} \to \Z/p \to 0, $$ which determines a $(d+1)$-dimensional  Kummer flag of $(\Gamma,1)$-bundles  $$\nabla_{d+1,1}:=V_{1,1} \subset V_{2,1} \subset \ldots \subset V_{d,1}\subset V_{d+1,1}.$$ Thus,  Theorem  \ref{ThmLiftK} implies the following.
\begin{coro}\label{CoroLiftK}
Let $r \geq 2$ be an integer. Assume that $\Z/p^r(1)\simeq \Z/p^r$.
Let $V$ be a $(\Gamma,r)$-bundle, which fits into a  Kummer $(\Gamma,r)$-flag  $$\nabla_{d,r}=V_{1,r} \subset V_{2,r} \subset \ldots \subset V_{d,r}=V.$$   Then, the natural map  \[ H^1(\Gamma, V) \to  H^1(\Gamma, V/p)\] is surjective.
\end{coro}

We go back to the proof of Theorem \ref{ThmLiftK}.

\begin{proof}
It is enough to prove the first statement of the Theorem: the second one follows by duality, given that the dual of a Kummer flag is again a Kummer flag, and taking duals swaps subobjects and quotients. \\
We proceed by induction on $d$. If $d=1$ or $d=2$, the proof is identical to that of Proposition \ref{propploye}.\\
Assume that $d > 2$ and that the result holds for all Kummer flags of dimension $< d$. Let $\nabla_{d,r} = (V_{i,r})_{1 \leq i \leq d}$ be a $d$-dimensional Kummer flag, and let $\nabla^\sharp_{d-1,r+1}$ be as in the statement of the first part of the Theorem.
    
  Consider the natural $(d-1)$-dimensional Kummer flags $\nabla_{d-1,r} \subset \nabla_{d, r}$ and $\nabla_{d/1,r} := \nabla_{d, r}/V_{1,r}$ associated to $\nabla_{d, r}$. 
  By induction, there exists a Kummer lift $\nabla_{d-1,r+1}^\flat$ of $\nabla_{d-1,r}$ compatible with $\nabla_{d-2,r+1}^{\sharp}$.
    
    \begin{itemize}[leftmargin=*]
        \item Assume first that there exists $2 \leq k \leq d-1$ such that $i_1(k)=0$, i.e. that the modulo $p$ extension 
    \[0 \to V_{k-1, 1} \to V_{k,1} \to L_{k, 1} \to 0\]
        splits. Since $\nabla_{d-1,r+1}^\flat$ is Kummer, the modulo $p^{r+1}$ extension 
        \[0 \to V_{k-1, r+1} \to V_{k,r+1} \to L_{k, r+1} \to 0\]
    splits. The choice of a section $s : L_{k,r+1} \to V_{k,r+1}$ determines a 1-dimensional direct factor $s(L_{k,r+1})$ of $\nabla_{k,r+1}^\flat$, hence 1-dimensional $(\Gamma,r+1)$-submodules of $\nabla^\sharp_{d-1,r}$ and of $\nabla_{d,r}$. In particular, we have a cartesian commutative diagram modulo $p^r$:
    \[
    \xymatrix{
    \nabla_{d,r} \ar[r] \ar[d] & \nabla^\sharp_{d-1, r} \ar[d] \\
    \nabla^{s}_{d-1,r} := \nabla_{d,r}/s(L_{k,r}) \ar[r] & \nabla^{\sharp,s}_{d-2,r} := \nabla^\sharp_{d-1, r} / s(L_{k,r}) \, .
    }
    \]
    By induction, there exists a lift $\nabla_{d-1, r+1}^s$ of $\nabla_{d-1,r}^s$ compatible with $\nabla_{d-2, r+1}^{\sharp, s} := \nabla_{d-1, r+1}^{\sharp} / s(L_{k,r+1})$. \\
    Then one defines a lift $\nabla_{d, r+1}$ of $\nabla_{d, r}$ as the following cartesian diagram
    \[
    \xymatrix{
    \nabla_{d,r+1} \ar[r] \ar[d] & \nabla^\sharp_{d-1, r+1} \ar[d] \\
    \nabla^{s}_{d-1,r+1} \ar[r] & \nabla^{\sharp,s}_{d-2,r+1} \, .
    }
    \]
    By Lemma \ref{lem 1 glue Kummer}, $\nabla_{d,r+1}$ is a Kummer lift of $\nabla_{d,r}$, extending $\nabla^\sharp_{d-1,r+1}$ (and $\nabla^\flat_{d-1,r+1}$).
        \item Assume now that $i_1(d)=0$, i.e. that the extension 
        \[0 \to V_{d-1, 1} \to V_{d,1} \to L_{d, 1} \to 0\]
        splits. Since $\nabla_{d,r}$ is Kummer, then the analogous extension splits modulo $p^r$. Then we define $\nabla_{d,r+1}$ as the direct sum $\nabla^\flat_{d-1,r+1} \oplus L_{d,r+1}$.
        Then $\nabla_{d,r+1}$ is the required Kummer lift of $\nabla_{d,r}$.
        
          \item Assume now that $i_1(d)=1$, i.e. that the extension 
        \[0 \to V_{d-1/1, 1} \to V_{d/1,1} \to L_{d, 1} \to 0\]
        splits. Since $\nabla^\sharp_{d-1,r+1}$ is Kummer, the analogous extension splits modulo $p^{r+1}$. The choice of a splitting defines a $2$-dimensional subflag $\nabla^!_{2,r}$ of $\nabla_{d,r}$, and $\nabla_{d,r}$ appears in the following pushout diagram of inclusions
        \[
        \xymatrix{
        \nabla_{1,r} \ar[r] \ar[d] & \nabla_{d-1, r} \ar[d] \\
        \nabla^!_{2,r} \ar[r] & \nabla_{d,r} \, .
        }
        \]
        By the $2$-dimensional case, the flag $\nabla^!_{2,r}$ lifts to a flag $\nabla^!_{2,r+1}$ compatible with both $\nabla^{\flat}_{1,r+1}=L_{1,r+1}$ and $L_{d,r+1}$, and one can define the required lift $\nabla_{d,r+1}$ by the following pushout diagram of inclusions
        \[
        \xymatrix{
        \nabla^\flat_{1,r+1} \ar[r] \ar[d] & \nabla^\flat_{d-1, r+1} \ar[d] \\
        \nabla^!_{2,r+1} \ar[r] & \nabla_{d,r+1} \, .
        }
        \]
        Finally, one checks that the flag $\nabla_{d,r+1}$ we just defined is Kummer and satisfies the statement of the Theorem.
          
          \item Assume finally that $i := i_1(d) \geq 2$ (and that we are not in the previous cases). \\
          Then the modulo $p$ extension 
    \[0 \to V_{d-1/i-1, 1} \to V_{d/i-1,1} \to L_{d, 1} \to 0\]
    does not split, while the extension
    \[0 \to V_{d-1/i, 1} \to V_{d/i,1} \to L_{d, 1} \to 0\]
    does split.
    In particular, the choice of a splitting of the second one defines a non-split extension
    \[
    0 \to L_{i,1} \to P_{i,d,1} \to L_{d,1} \to 0 \, .
    \]
By assumption $i_1(i) \neq 0$,  so that the following extension does not split:
    \[0 \to V_{i-1, 1} \to V_{i,1} \to L_{i, 1} \to 0.\]
    The argument is now very similar to the proof of Proposition \ref{propploye}. \\
    The obstruction to glue $\nabla^{\flat}_{d-1,r+1}$ and $\nabla^{\sharp}_{d-1,r+1}$ to a lift $\nabla_{d,r+1}$ of $\nabla_{d,r}$, is a class $$c_1 \in H^2(\Gamma, L_{d,1}^\vee \otimes L_{1,1})$$ (see the discussion after Definition \ref{defigluift}). If $c_1=0$, then  gluifting can be done, and $\nabla_{d,r+1}$ is automatically a Kummer flag, compatible with  $\nabla^{\sharp}_{d-1,r+1}$. \\
   Assume that $c_1 \neq 0$. We are going to modify $\nabla^{\flat}_{d-1,r+1}$, so that gluifting becomes possible. \\ 
    By assumption, the extension 
   \[0 \to V_{d-1/i-1, 1} \to V_{d/i-1,1} \to L_{d, 1} \to 0\]
   does not split, nor the twist of its dual by $L_{1,1}$:
   \[0 \to L_{d,1}^\vee \otimes L_{1,1} \to V_{d/i-1,1}^\vee \otimes L_{1,1} \to V_{d-1/i-1, 1}^\vee \otimes L_{1, 1} \to 0 \, .\] Denote its class by $$p_1 \in \Ext^1_\Gamma(V_{d-1/i-1, 1}^\vee, L_{d,1}^\vee) \simeq H^1(\Gamma, L_{d,1}^\vee \otimes \left(V_{d-1/i-1, 1}\right)) \, .$$ 
   Since $p_1\neq 0$, Lemma \ref{LemManaLift} implies that there exists a class $$\epsilon^{\flat} \in H^1(\Gamma, V_{d-1/i-1, 1}^\vee \otimes L_{1, 1}) = \Ext^1_\Gamma(V_{d-1/i-1, 1}, L_{1, 1}) \, ,$$ such that $p_1 \cup \epsilon^{\flat} = c_1$. 

   Consider the extension $i_* \epsilon^{\flat} \in \Ext^1_\Gamma(V_{d-1/i-1, r+1}, L_{1, r+1})$ and its pull-back $\widetilde{i_* \epsilon^{\flat}} \in \Ext^1_\Gamma(V_{d-1/1, r+1}, L_{1, r+1})$ by the  surjection $V_{d-1/1,r+1} \to V_{d-1/i-1, r+1}$. Define $V^{\flat, \mathrm{ad} }_{d-1,r+1} \in \Ext^1_\Gamma(V_{d-1/1, r+1}, L_{1, r+1})$ as the Baer sum of the extensions $V^\flat_{d-1,r+1}$ and the opposite of $\widetilde{i_* \epsilon^{\flat}}$ in this Ext group. Then $V^{\flat, \mathrm{ad} }_{d-1,r+1}$ is a new lift of $\nabla^{\flat}_{d-1,r}$, arising as a small deformation of $\nabla^{\flat}_{d-1,r+1}$. Note that $\nabla^{\flat, \mathrm{ad} }_{d-1,r+1}$ is Kummer since none of the $i_1(k)$ are zero, for $2 \leq k \leq d-1$. Denote by  $$c_1^{\mathrm{ad}} \in H^2(\Gamma, L_{d,1}^\vee \otimes L_{1,1})$$  the obstruction to gluing $\nabla^{\sharp}_{d-1,r+1}$ and $\nabla^{\flat, \mathrm{ad}}_{d-1,r+1}$ to a lift $\nabla_{d,r+1}$ of $\nabla_{d,r}$. Using the explicit constructions of the obstruction classes $c_1$ and $c_1^{\mathrm{ad}}$ in Lemma \ref{lem gluifting obs}, one computes: $$c_1^{\mathrm{ad}}=c_1-p_1 \cup \epsilon^{\flat}=0. $$
Therefore, $\nabla^{\flat, \mathrm{ad}}_{d-1,r+1}$ and $\nabla^{\sharp}_{d-1,r+1}$ glue, to the desired $\nabla_{d,r+1}$, that is Kummer by Lemma \ref{lem 2 glue Kummer}.\qedhere
\end{itemize}\end{proof}

One could hope for a statement generalizing both Proposition \ref{propploye} and Theorem \ref{ThmLiftK}, with no assumptions on roots of unity in the arithmetic case and with a suitable notion of Kummer flag. This cannot be achieved using only cyclotomic twists for the $L_k$'s, as shown in the following example.
\begin{ex}\label{ex5dim}
Let $K = \Q_2$ (resp. $\Q_\ell$, with $\ell \equiv 3 \, \pmod{4}$) and $p=2$. The map $$\alpha : H^1(K, \Z/4\Z) \to H^1(K, \Z/2\Z),$$ induced by reduction modulo $2$, is not surjective: its image is a subgroup of index $2$. More precisely, identifying $H^1(K, \Z/2\Z)$ with $K^\times / (K^\times)^2$, it is classical that for all $x \in K^\times$, $(x) \in H^1(K, \Z/2\Z)$ lifts to $H^1(K, \Z / 4\Z)$ if and only if $$(x) \cup (x) = (x) \cup (-1)=0 \in H^2(K, \Z / 2\Z).$$ Hence  $H^1(K, \Z/2\Z) \setminus \im(\alpha) = \{(-1),(-2),(-5),(-10)\}$ (resp. $H^1(K, \Z/2\Z) \setminus \im(\alpha) = \{(-\ell),(\ell)\}$).

Let $\varepsilon := (-2)$ and $\varepsilon' := (-5)$ in $H^1(K, \Z / 2 \Z)$ (resp. $\varepsilon := (-\ell)$ and $\varepsilon':=(\ell)$). Then $\varepsilon, \varepsilon' \notin \im(\alpha)$ and $\varepsilon \cup \varepsilon' = 0$.

The relation $\varepsilon \cup \varepsilon' = 0$ implies that the representations of $\Gamma_K$ defined by $\rho_1 := \left(\begin{array}{ccc} 1 & \varepsilon &  \varepsilon \\ 0 & 1 & 0 \\ 0 & 0 & 1 \end{array}\right)$ and $\rho_2 := \left(\begin{array}{ccc} 1 & 0 &  \varepsilon' \\ 0 & 1 & 0 \\ 0 & 0 & 1 \end{array}\right)$ glue to a four-dimensional representation $\rho_3 := \left(\begin{array}{cccc} 1 & \varepsilon &  \varepsilon & \ast \\ 0 & 1 & 0 & \varepsilon' \\ 0 & 0 & 1 & 0 \\ 0 & 0 & 0 & 1 \end{array}\right)$. Similarly, the representations $\rho_2$ and $\rho_4 := \left(\begin{array}{ccc} 1 & 0 &  \varepsilon' \\ 0 & 1 & \varepsilon \\ 0 & 0 & 1 \end{array}\right)$ glue to a four-dimensional representation $\rho_5 := \left(\begin{array}{cccc} 1 & 0 & \varepsilon' & \ast \\ 0 & 1 & 0 & \varepsilon' \\ 0 & 0 & 1 & \varepsilon \\ 0 & 0 & 0 & 1 \end{array}\right)$.

Finally, the obstruction to glue $\rho_3$ and $\rho_5$ (along $\rho_2$) to a five-dimensional representation lies in $H^2(K, \Z/2\Z)$. Up to modifying the top right hand side coefficient of $\rho_3$ by an element in $H^1(K, \Z / 2\Z)$, one can assume that this obstruction is trivial, hence $\rho_3$ and $\rho_5$ glue to a representation $\bar \rho : \Gamma_K \to \GL_5(\F_2)$ defined by
\[\bar \rho := \left(\begin{array}{ccccc} 1 & \varepsilon &  \varepsilon & \ast & \ast \\ 0 & 1 & 0 & \varepsilon' & \ast \\ 0 & 0 & 1 & 0 & \varepsilon' \\ 0 & 0 & 0 & 1 & \varepsilon \\ 0 & 0 & 0 & 0 & 1 \end{array}\right)\]
and we denote by $\nabla_{5,1}$ the associated flag.

Then $\bar \rho$ does not lift to a representation $\rho : \Gamma_K \to \GL_5(\Z / 4 \Z)$ of the  shape:
\[\rho = \left(\begin{array}{ccccc} \chi_1 & \ast &  \ast & \ast & \ast \\ 0 & \chi_2 & 0 & \ast & \ast \\ 0 & 0 & \chi_3 & 0 & \ast \\ 0 & 0 & 0 & \chi_4 & \ast \\ 0 & 0 & 0 & 0 & \chi_5 \end{array}\right)\]
such that $\chi_i$ are powers of the cyclotomic character.
Equivalently,  $\nabla_{5,1}$ does not lift to a $(\Gamma_K,2)$-flag $\nabla_{5,2}$,   such that $L_{k,2}$ are powers of the cyclotomic character and the extensions $V_{3/2,2}$ and $V_{4/3,2}$ split.\\
Indeed, assume the contrary and denote $\chi_i$ by $\chi^{\xi_i}$ where $\chi$ is the cyclotomic character modulo $4$ and $\xi_i \in \2$. One sees that $\varepsilon, \varepsilon' \in H^1(K, \Z/2\Z)$ then lift to $$H^1(K,\chi_j^\vee \otimes \chi_i) = H^1(K, L_{j,2}^\vee \otimes L_{i,2}) = H^1(K, \Z/4\Z(\xi_i+\xi_j)),$$ for all $(i,j) \in \{(1,2), (1,3), (2,4), (3,5), (4,5)\}$, hence for all such $(i,j)$, one has $\xi_i + \xi_j = 1$, i.e. $\xi_i \neq \xi_j$. Therefore $\xi_1 \neq \xi_2$, $\xi_2 \neq \xi_4$, $\xi_4 \neq \xi_5$, hence $\xi_1 \neq \xi_5$, and $\xi_1 \neq \xi_3$, $\xi_3 \neq \xi_5$, hence $\xi_1 = \xi_5$. So we get the contradiction $\xi_5 \neq \xi_5$.
\end{ex}

\begin{rem}
It follows from \cite[Theorem 6.4.4]{EG} that, in the case $\ell=p=2$, $\ovl\rho$ as in Example \ref{ex5dim} admits a lift to a completely reducible representation $\rho\colon \Gamma_F\to\GL_5(\Z/4\Z)$. However, it is important to the Emerton-Gee method to allow the characters on the diagonal to be arbitrary crystalline characters. For instance if $F=\Q_p$, crystalline characters are unramified twists of powers of the cyclotomic character, and the images of Frobenius under the unramified twists appearing on the diagonal give distinct variables on the Emerton-Gee stack, that provide the necessary freedom to construct lifts for all possible $\ovl\rho$. \\
In the case when $\ell\ne p$, one can deduce from \cite[Section 2.4.4]{CHT} that $\ovl\rho$ admits a lift $\rho\colon\Gamma_F\to\GL_5(\Z/4\Z)$ equipped with a complete flag, and such that the inertia subgroup of $\Gamma_F$ acts unipotently. Then a simple argument shows that the flag for $\rho$ cannot satisfy condition (1) of Definition \ref{def kummer}, in view of the example above.
\end{rem}

    From Proposition \ref{propploye} and Theorem \ref{ThmLiftK}, one  derives the following  much weaker 

\begin{coro}\label{CoroLift}
   Let $\Gamma$ be one of the following groups: 
   \begin{itemize}
       \item the absolute Galois group of a local field $F$ (a finite extension of $\Q_\ell$ or $\F_\ell((t))$, with $\ell=p$ allowed.)
       \item   the topological fundamental group of a closed connected  orientable  surface.
   \end{itemize}
   Let $d \geq 1$ be an integer, and let $$\rho_1: \Gamma \to \GL_d(\F_p)$$ be a mod $p$ representation. The following holds. \begin{enumerate}
        \item {Assume that the representation $\rho_1$ has a unique complete $\Gamma$-invariant flag. Together with this complete flag, it then lifts for all $r \geq 2$ to a representation $$\rho_r: \Gamma \to\GL_d(\Z/p^r \Z),$$ and these lifts can be chosen in such a way that $\rho_{r+1}$ reduces to $\rho_r$ mod $p^r$ for every $r$.}
            \item {In the arithmetic case, assume that $F$ contains all $p^2$-th roots of unity. \\Then $\rho_1$ lifts to a representation $$\rho_2: \Gamma \to \GL_d(\Z/p^2) .$$ }
        
    \end{enumerate} 
\end{coro}


\begin{dem}
In the topological case, we first extend $\rho_1$ to a continuous representation $\widehat{\rho}_1 : \widehat{\Gamma} \to \GL_d(\F_p)$. Lifting $\widehat{\rho}_1$ to continuous representations of $\widehat{\Gamma}$ implies lifting $\rho_1$ to representations of $\Gamma$ (see Remark \ref{rem completion}). From now on, $\Gamma$ denotes the profinite completion $\widehat{\Gamma}$.

In Item (1), observe that the unique flag for $\rho_1$ is necessarily wound (see Lemma \ref{remwoundunique}). The result then follows from the step-by-step liftability statement of Corollary \ref{CoroLiftWound}. Let us prove Item (2), using  a standard technique from group cohomology. 
Denote by $\Gamma_0 \subset  \Gamma $ the open subgroup $\Ker(\rho_1)$. Set $\Gamma^0:=\Gamma  /\Gamma_0$.
    Let $S \subset \Gamma^0$ be a $p$-Sylow subgroup and $\Gamma_S \subset \Gamma$ be the pullback of $S$ in $\Gamma$ (it is a prime-to-$p$ finite index subgroup). By passing to the quotient and restriction, $\rho_1$ gives rise to a representation $\Gamma_S \to \GL_d(\F_p)$. Since $S$ is a $p$-group, this representation is conjugate to a strictly upper triangular representation $\Gamma_S \to \mathbf U_d(\F_p)$. Define $V_d:=\F_p^d$, viewed as a $(\Gamma_S,1)$-module via  $\rho_{1\vert \Gamma_S}: \Gamma_S \to \GL_d(\F_p)$. By Remark \ref{RemKum1}, it follows that $V_d$ can be equipped with a Kummer flag $\nabla_{d,1}$.  Applying  Theorem \ref{ThmLiftK} (to $\Gamma_S$ and $r=1$) one gets that $\nabla_{d,1}$ lifts mod $p^2$. One then concludes using \cite{DCF0}, Lemmas 3.2 and 3.4.
\end{dem}

\medskip
\begin{rem}

No assumption on roots of unity is required in Corollary \ref{CoroLift}(1), thanks to the (strong) assumption that $\rho_1$ admits a unique complete $\Gamma$-invariant flag. Also, in Corollary \ref{CoroLift}(2), one can replace the assumption on roots of unity, by the weaker condition $F(\mu_p) = F(\mu_{p^2})$, see Remark \ref{rem kummer general}.
\end{rem}
\begin{rem}
   In the topological case, it would be interesting to provide an alternative elementary proof of item (2) of the Corollary, using the presentation of  $\pi_1(S)$ by $2g$ generators $X_i, Y_i$, with  the single relation   $[X_1,Y_1] \ldots [X_g,Y_g]=1$. To our knowledge, such a proof is not available in the literature.  Note that, in genus $g=1$, item (2) boils down to a nice exercise: prove that two commuting  elements of $\GL_d(\F_p)$ admit lifts, to  commuting  elements of $\GL_d(\Z_p)$. Note also that item (2), in the  arithmetic case, is proved in \cite{Bo} using the presentation of $\Gamma$ by generators and relations. This approach is much more involved than ours, but works without assumptions on roots of unity. 
\end{rem}

\begin{rem}
   As already mentioned in Section \ref{SectGen}, the preceding Corollary is still true upon replacing $\F_p$ by a field $k$ of characteristic $p$, and $\Z/p^2$ (resp. $\Z_p$) by $\W_2(k)$ (resp. $\W(k)$). The proof  adapts, with minor modifications.
\end{rem}
\begin{rem}
    For Galois groups of local fields,  the proof of  Corollary \ref{CoroLift} is altogether extremely short. By  \cite{EG} or \cite{BIP}, item (1) actually holds unconditionally (dismissing flags).   We hope this can be done by an elementary treatment, as well. 
\end{rem}

\section*{Acknowledgments} 
We are grateful to Julien Marché, who brought to our attention that the method we initially designed to study liftability of   Galois representations of local fields, applies verbatim to representations of fundamental groups of surfaces. Our thanks go to the anonymous referee, and to Gebhard Böckle, Pierre Colmez, Hélène Esnault, Jean-Pierre Serre and Olivier Wittenberg. They made meaningful remarks and helped improve the exposition. \\This work was completed while Mathieu Florence was visiting Nesin Mathematics Village, in \c{S}irince.   He warmly thanks the staff for their kindness and excellent working conditions.

\end{document}